\documentclass[11pt]{amsart}

\usepackage{amssymb,amsmath}
\usepackage[]{graphicx}
\usepackage{amssymb, epsfig}
\usepackage{color}
\numberwithin{figure}{section} \numberwithin{equation}{section}

\newcommand{\ep}{\varepsilon}
\newcommand{\R}{\mathbb{R}}
\newcommand{\tildezero}{\gamma _0 ^\ep}
\newcommand{\tildet}{\gamma _t ^\ep}

\newcommand{\tildeint}{\Omega _t ^\ep}
\newcommand{\tildeext}{\Omega \setminus \overline{\tildeint}}
\newcommand{\EB}{e^{-\beta t/\ep ^2}}

\newtheorem{theorem}{Theorem}[section]
\newtheorem{lemma}[theorem]{Lemma}
\newtheorem{corollary}[theorem]{Corollary}

\newtheorem{prop}[theorem]{Proposition}

\usepackage{epsf,pstricks,multicol,pst-grad,color,pst-node,pst-text,pst-3d,pst-tree}
\begin{document}
%
%
%
%
%
\title[Interfaces for the Stochastic Allen--Cahn equation]
{
Generation
of fine transition layers and their dynamics
    for the stochastic Allen--Cahn equation}

%

%
\author{M. Alfaro}
\address{Matthieu Alfaro, IMAG, Univ. Montpellier, CNRS, Montpellier, France.} \email{matthieu.alfaro@umontpellier.fr}

\author{D. Antonopoulou}
\address{Dimitra Antonopoulou, Department of Mathematics, University of
Chester, Thornton Science Park, CH2 4NU, Chester, UK, and,
Institute of Applied and Computational Mathematics, FORTH,
GR--711 10 Heraklion, Greece.}
\email{d.antonopoulou@chester.ac.uk}

\author{G. Karali}
\address{Georgia Karali, Department of Mathematics and Applied Mathematics, University of Crete,
GR--714 09 Heraklion, Greece, and, Institute of Applied and
Computational Mathematics, FORTH, GR--711 10 Heraklion, Greece.}
\email{gkarali@uoc.gr}

 \author{H. Matano}
 \address
{Hiroshi Matano, Meiji Institute for Advanced Study of Mathematical Sciences, Meiji University, 4-21-1 Nakano, Tokyo 164-8525, Japan.}
\email{matano@meiji.ac.jp}

%
%

\begin{abstract}
We study an $\ep$-dependent stochastic Allen--Cahn equation with a
mild random noise on a bounded domain in $\mathbb{R}^n$, $n\geq
2$. Here $\ep$ is a small positive parameter that represents
formally the thickness of the solution interface, while the mild
noise $\xi^\ep(t)$ is a smooth random function of $t$ of order
$\mathcal O(\ep^{-\gamma})$ with $0<\gamma<1/3$ that converges to
white noise as $\ep\rightarrow 0^+$. We consider initial data that
are independent of $\ep$ satisfying some non-degeneracy
conditions, and prove that steep transition layers---or
interfaces---develop within a very short time of order
$\ep^2|\ln\ep|$, which we call the ``generation of interface".
Next we study the motion of those transition layers and derive a
stochastic motion law for the sharp interface limit as
$\ep\rightarrow 0^+$. Furthermore, we prove that the thickness of
the interface for $\ep$ small is indeed of order $\mathcal O(\ep)$
and that the solution profile near the interface remains close to
that of a (squeezed) travelling wave; this means that the
presence of the noise does not destroy the solution profile near
the interface  as long as the noise is spatially uniform.
Our results on the motion of interface improve the earlier
results of Funaki (1999) and Weber (2010) by considerably
weakening the requirements for the initial data and establishing
the robustness of the solution profile near the interface that
has not been known before.
\end{abstract}
%
%
 \maketitle


 \pagestyle{myheadings}
\thispagestyle{plain}
%
%
%
\section{Introduction}\label{s:intro}

We consider a stochastic
 Allen--Cahn equation with  a Neumann boundary condition
\begin{equation}\label{AC}
\begin{split}
&\partial _t u= \Delta u + \frac{1}{\ep^2} f(u)+\frac{1}{\ep}
\xi^\ep(t),
\;\;\;t>0,\;\;\;x \in\Omega,\\
&\frac{\partial u}{\partial \nu}=0,\;\;\;t>0,\;\;\;x\in\partial\Omega,\\
&u(x,0)=u_0(x),\;\;\;x\in\Omega,
\end{split}
\end{equation}
where $\Omega$ is a smooth open bounded domain in $\mathbb{R}^n$
($n\geq 2$), $\nu$ is the
outward unit normal vector to
$\partial \Omega$, and $\ep>0$ is a small parameter.
The nonlinearity $f$ is of the {\it bistable} type, and the
perturbation term $\xi ^\ep (t)$ is what we call a mild noise
which is a smooth but random function of $t$ that behaves like
 an irregular white noise in the limit as $\ep  \to 0$.
As mentioned in \cite{Fun99}, such an equation can be viewed
as describing intermediate (mesoscopic) level phenomena
between macroscopic and microscopic ones. In such a scale, an
active noise appears as a correction term to the reaction-diffusion
equation when fluctuations in the hydrodynamic limit is
taken into account, see \cite{Spohn}.

Our main goal is to make a detailed analysis of the
{\it sharp interface limit} of the problem \eqref{AC} as $\ep\to 0$.
In the deterministic case where the perturbation term
$\xi^\ep$ is replaced by non-random, uniformly bounded
smooth functions, the sharp interface limit of \eqref{AC} is
well understood: it is known that the solution $u$ typically
develops steep transition layers---or {\it interfaces}---of
thickness ${\mathcal O}(\ep)$ within a very short time,
which we call the {\it generation of interface} (or one may
call it the {\it emergence of transition layers}).
Furthermore, as $\ep\to 0$,
those layers converge to interfaces of thickness $0$ whose
law of motion is given by the curvature flow with
some driving force (the {\it propagation of interface}).
See \cite{ahm, C1} and the references therein for details.

 In the present problem, the perturbation term $\xi^\ep(t)$
is random while it is no longer uniformly bounded as $\ep\to 0$,
since it converges to white
 noise in a certain sense. This makes the analysis harder than  in
the classical deterministic case. Nonetheless, as we shall see,
it is possible to derive a number of results that are as optimal
as those established for the classical deterministic problem.

Our first result concerns the {\it generation of interface}. More
precisely, we consider solutions of \eqref{AC} with
$\ep$-independent initial data, and show that steep transition
layers of thickness ${\mathcal O}(\ep)$ emerge within a very
short time. This thickness estimate of order ${\mathcal O}(\ep)$
is the same optimal estimate known for the classical
deterministic problem. Next, we discuss the {\it propagation of
interface} and show that the thickness of the layer remains of
order ${\mathcal O}(\ep)$ as time passes, and that in the
sharp interface limit, as $\ep\to 0$, the law of motion of the
interface is given by
\begin{equation}\label{mmc-debut}
V=(n-1)\kappa + c \dot{W}_t,
\end{equation}
where $V$ is the inward normal velocity, $\kappa$ denotes the
 mean curvature, 
$c$ is a positive constant and $\dot{W}_t$ is a white noise. The
above equation was first derived in \cite{Fun99, weber1} for a
special class of $\ep$-dependent initial data that already have
well-developed transition layers. Our result confirms the
validity of the same equation for rather general $\ep$-independent
initial data. Furthermore, we also show that the {\it profile of
the solution} near the interface is well approximated by  a
traveling wave. This implies that the solution profile near the
interface is quite robust and is not destroyed by the random
noise, as long as the noise depends only on the time variable.

The singular limit of a stochastic Allen--Cahn equation of the
form \eqref{AC} was studied by Funaki in his pioneering work
\cite{Fun99} for two space dimensions, and later by Weber
\cite{weber1} for general space dimensions $n\geq 2$. Our results
improve the work of \cite{Fun99, weber1} in three notable
aspects. First, as mentioned above, our paper studies the
emergence of steep transition layers (the {\it generation of
interface}) at the very initial stage of evolution, which is not
discussed in \cite{Fun99, weber1}. Secondly, our ${\mathcal
O}(\ep)$ estimate of the thickness of layers is optimal and
therefore, is considerably better than the order ${\mathcal
O}(\ep^\alpha)\,(0<\alpha<1)$ estimates presented in
\cite{Fun99, weber1}. Thirdly, we show the robustness of the
solution profile around the interface in the presence of noise
({\it rigidity of profile}), a fact that has been totally unknown
before.

Concerning results on the generation of interface, let us also mention the very recent papers
\cite{Lee-2015, Lee-2016}. In \cite{Lee-2015}, the author considers the one-dimensional case
with space-time white noise and studies both the generation and motion of the interface, thus improving the work
\cite{Fun95}, which did not consider the generation of interface. However, as we shall explain in Subsection
\ref{ss:deterministic-stochastic}, the one-dimensional case is totally  different from the multi-dimensional
case as the curvature effect does not appear in the former. Therefore the problems treated in \cite{Fun95, Lee-2015}
are different from the subject of the present paper. In \cite{Lee-2016}, the authors considers a multi-dimensional
problem under a space-time noise that is smooth in $x$. However \cite{Lee-2016} deals with only the generation of
interface, thus the motion of interface under such a noise remains unknown.

We also refer to the influential theory of stochastic
viscosity solutions of Lions and Souganidis which covers a large
class of stochastic fully nonlinear partial differential
equations with applications to phase transitions and propagation
of fronts in the presence of noise; see for example the works
\cite{Lio-Sou-98, Lio-Sou-98-bis, Lio-Sou-00, Lio-Sou-00-bis}. In \cite{Lio-Sou-98}, the authors introduced the
notion of weak solutions (stochastic viscosity solutions) for
parabolic, possibly degenerate, second-order stochastic pdes
posed in $\mathbb{R}^N$. In particular \cite[subsection 2.3]{Lio-Sou-98-bis}
 considers the specific case of the
$\varepsilon$-dependent Allen-Cahn equation with the stochastic
perturbation introduced by Funaki in \cite{Fun99}, posed in the
unbounded domain $\mathbb{R}^N$.  It is supplemented with a general initial data $u_0$, not depending on $\ep$ and with not
necessarily convex initial interface
$\Gamma_0:=\left\{x\in\mathbb{R}^N:\;u_0(x)=0\right\}$, see \cite[(2.6) and (2.7)]{Lio-Sou-98-bis}. In the context of viscosity solutions, a proper
approximation of the stochastic problem is proposed which, on the
sharp interface limit, yields the stochastic motion by mean
curvature and the limiting profile of $u$ to $\pm 1$. Notice however that our approach stands in the
short-time existence of a solution to \eqref{mmc-debut} as proposed in \cite{Fun99} and \cite{DLN}, see Section \ref{s:mmc}. This allows us to prove much finer properties of the convergence of \eqref{AC} to \eqref{mmc-debut}, namely the optimal thickness estimate of the thin layers of solutions of \eqref{AC} when $\ep$ is very small (see \eqref{resultat} of Theorem~\ref{th:main}), and the proof of the robustness of the layer profile (see Theorem~\ref{th:profile}), neither of which can be obtained through the viscosity approach, see \cite{Alf-Dro-Mat-12}.



\subsection{Assumptions}\label{ss:assumptions}

Let us state our standing assumptions in the present paper.
The nonlinearity is given by $f(u):=-W'(u)$, where $W(u)$ is a
double-well potential with equal well-depth, taking its global
minimum value at $u=a_ \pm$. More specifically, we assume that
\begin{equation}\label{cf0}
f\ \ \mbox{is}\ \ C^2\ \ \mbox{and has exactly three zeros} \ \
a_-<a<a_+,
\end{equation}
\begin{equation}\label{cf1}
f'(a_\pm)<0,\ \ f'(a)>0,
\end{equation}
and
\begin{equation}\label{equalwell}
\int_{a_-}^{a_+}f(u)du=0.
\end{equation}
This last assumption \eqref{equalwell} makes $f$ a {\it balanced}
bistable nonlinearity. We will use this assumption only in
Section \ref{s:motion}, where we study the propagation of
interface. No such assumption is needed for the emergence of
interface, which we discuss in Section \ref{s:emergence}.


Concerning the initial data, we assume that $u_0\in
C^2(\overline{\Omega})$, and define
\begin{equation}\label{cu0}
C_0:=\|u_0\|_{C^0(\overline{\Omega})}+\|\nabla
u_0\|_{C^0(\overline{\Omega})}+\|\Delta
u_0\|_{C^0(\overline{\Omega})}.
\end{equation}
The initial interface is defined by
\begin{equation}\label{def:initial-interface}
\Gamma_0:=\left\{x\in\Omega:\;u_0(x)=a\right\}.
\end{equation}
We assume that $\Gamma _0 \subset \subset \Omega$ is a
$C^{2,\alpha}$ ($0<\alpha<1$) hypersurface without boundary
and that
\begin{equation}\label{hyp:gradient}
\nabla u_0(x)\cdot n(x)\neq 0\; \text{ for any } x\in\Gamma_0,
\end{equation}
where $n=n(x)$ denotes the outward unit normal vector to
$\Gamma_0$ at $x$.

Let $\Omega_0$ denote the region enclosed by $\Gamma_0$. Without loss
of generality, we may assume that
\begin{equation}\label{def:interior}
u_0(x)<a\;\mbox{ for any }x\in \Omega_0\ \ \mbox{and} \ \
u_0(x)>a\;\mbox{ for any }x\in \Omega\setminus\overline{\Omega_0}.
\end{equation}


As regards the perturbation term $\xi^\ep(t)$, we shall consider
two types of mild noises as specified below, following \cite{Fun99} and \cite{weber1}.\\

\underline{\textit{First type of noise (MN1)}}\\

Following Funaki \cite{Fun99}, we consider a mild noise
$\xi^\ep$ given in the form
\begin{equation}\label{1}
\xi^\ep(t):=\ep^{-\gamma _1}\xi(\ep^{-2\gamma _1}t),\;\;\;t>0,
\end{equation}
for some
\begin{equation}\label{gamma1}
0<\gamma _1<\frac{1}{3},
\end{equation}
 where $\xi(t)=\xi_t$ is a
stochastic process in $t$ that is
stationary and strongly mixing. More specifically, let $F_{\xi_{t_1 + \tau}, \ldots,
\xi_{t_k + \tau}}$ be the distribution function of the $k$ random variables $\xi_{t_1 + \tau}, \ldots,
\xi_{t_k + \tau}$, then the stochastic process
$\xi_t$ is called \textit{stationary} if for all $k$, $\tau$ and
for all $t_1, \ldots, t_k$
$$
F_{\xi_{t_1+\tau} ,\ldots, \xi_{t_k+\tau}}=
F_{\xi_{t_1},\ldots, \xi_{t_k}}.
$$
Let
$(\Omega_{prob},\mathcal{F},\mathbb P)$ be the probability space where
$\xi_t$ is realized, with $\mathcal{F}:=\sigma(\xi_r:\;0\leq r<
+\infty)$ the $\sigma$-algebra generated by $\xi_{r}$ for $0\leq
r< +\infty$, and $\mathbb P$ the probability measure. Then
$\mathcal{F}_{s,t}:=\sigma(\xi_r:\;s\leq r\leq t)$ is the subalgebra
of $\mathcal{F}$ generated by $\xi_{r}$ for $s\leq r\leq t$.
We assume that the process $\xi_t$ is strongly mixing in the
following sense: the mixing rate $\rho(t)$ defined by
\[
\rho(t):=\sup_{s\geq 0}\,\sup_{A\in
\mathcal{F}_{s+t,\infty},\;B\in\mathcal{F}_{0,s}}|\mathbb P(A \cap B) -
\mathbb P(A)\mathbb P(B)|/\mathbb P(B), \quad t\geq 0,
\]
satisfies
\[
\int_0^{\infty}\rho(t)^{1/p}dt<+\infty \quad\ \hbox{for some} \ \ p>3/2.
\]
In Funaki \cite{Fun99}, this last condition is used to derive
some estimates that are uniform in $\ep$;
see the proof of Proposition 4.1 and
Lemma 5.3 in \cite{Fun99}.


Furthermore, it is assumed that $t\mapsto \xi(t) \mbox{ is } C^1$
almost surely,
$$
|\xi(t)|\leq M, \quad\  |\dot{\xi}(t)|\leq M, \quad\ E[\xi
(t)]=0,
$$
for some deterministic constant $M$, with $\dot{\xi}:=\frac{d\xi}{dt}$.
Obviously, the above implies that
$$
t\mapsto \xi ^\ep(t) \mbox{ is } C^1 \mbox{ almost surely},
$$
and that
\begin{equation}\label{estimates}
\left|\xi^\ep(t)\right|\leq M\ep^{-\gamma _1}, \quad
|\dot{\xi^\ep}(t)|\leq M\ep^{-3\gamma _1}.
\end{equation}
In Funaki \cite{Fun99}, these conditions are used to justify the
limit interface equation \eqref{mmc} as $\ep\to 0$, but, as we
shall see, the estimate \eqref{estimates} will also
be fundamental for our analysis of  the initial formation of layers (the generation of interface).

Notice that
the coefficient $\ep^{-\gamma_1}$ in the definition
\eqref{1}
implies that $\xi^\ep(t)$ is unbounded as $\ep \to 0$.
As shown in \cite{Fun99}, $\xi^\ep(t)$ converges to an irregular
white noise as $\ep\to 0$ in a certain sense.


\vskip 10pt
\underline{\textit{Second type of noise (MN2)}}\\

Following Weber \cite{weber1}, we define the mild noise
$\xi^\ep(t)=\xi_t^\ep$ as the derivative of a mollified Brownian
motion. More precisely, let $W(t)$, $t\geq 0$, be a Brownian
motion defined on the space $(\Omega_{prob},\mathcal F,
\mathbb{P})$. (Here, as usual, the dependence of $W$ on the sample
points $\omega\in\Omega_{prob}$ is not shown explicitly.) For
technical reasons, $W(t)$ is extended over $\R$ by considering an
independent Brownian motion $\widetilde W(t)$, $t\geq 0$, and
setting $W(t)=\widetilde W(-t)$ for $t<0$. Then $W(t)$, $t\in
\mathbb{R}$, is a Gaussian process, with independent stationary
increments and a distinguished point $W(0)=0$ almost surely. Also,
let $\rho:\mathbb{R}\rightarrow \mathbb{R}_+$ be a mollifying
smooth and symmetric kernel, with $\rho=0$ outside $[-1,1]$ and
$\int_\R\rho=1$. The approximated Brownian motion $W^\ep(t), t\geq
0,$ is defined as usual by
\begin{equation}\label{mollifying}
W^\ep (t):=W\ast\rho^\ep(t):=\int_{-\infty}^{\infty}\rho^\ep(t-s)
W(s)ds,
\end{equation}
where $\rho^\ep(\tau):=\ep^{-\gamma _2} \rho(\ep^{-\gamma _2}\tau)$
for some constant $\gamma_2$ satisfying
\begin{equation}\label{gamma2}
0<\gamma _2<\frac{2}{3}.
\end{equation}
Note that the Brownian motion for negative times is needed only in
the expression \eqref{mollifying}, so only the negative times in
$(-\ep^{\gamma_2}, 0]$ will play a role.  The constant $\gamma_2$
determines how quickly $W^\ep$ converges to the true integrated
white noise as $\ep\to 0$. Since $W(t)$ is H\"older continuous
almost surely, $W^\ep(t)$ is a smooth function of $t$ almost
surely. The noise $\xi^\ep(t)$ is then defined as the derivative
of $W^\ep(t)$, that is,
\begin{equation}
\label{def:mild-noise} \xi^\ep(t)=\dot W^\ep(t).
\end{equation}

In \cite[Propositions 1.2 and 1.3]{weber1},
the author
derives estimates for $\xi^\ep(t)$ and its derivative $\dot\xi^\ep(t)$
in the form
\[
\left|\xi^\ep(t)\right|\leq M \ep^{-\tilde\gamma/2}, \quad\ \
|\dot{\xi^\ep}(t)|\leq M \ep^{-3\tilde\gamma/2} \qquad
(\gamma_2<{}^\forall \tilde\gamma<2/3),
\]
by using L\'evy's well-known result on the modulus of continuity
of Brownian motion:
\[
{\mathbb P}\left[\limsup_{\delta\to 0}\frac{1}{g(\delta)}
\max_{\underset{t-s\leq \delta}{0\leq s<t\leq T}}|W(t)-W(s)|=1\right]=1,
\]
where the modulus of continuity is given by
$g(\delta)=\sqrt{2\delta\log(\frac{1}{\delta})}$. Actually the
very same argument as in \cite{weber1} gives the following
slightly more refined estimates, whose proof is omitted
as it is straightforward---roughly speaking it suffices to
set $\delta = \ep^{\gamma _2}$ in $g(\delta)$.


\begin{prop}[Estimates of the noise term] \label{prop:Weber-noise}
For any $T>0$, there exist a non-random constant $M=M(T)>0$ and
(random) $\ep_0>0$ such that, for all $0<\ep\leq \ep_0$ and all
$0\leq t\leq T$,
\begin{equation}\label{estimates2}
\left|\xi^\ep(t)\right|\leq M \ep^{-\gamma _2/2}\vert \log
\ep\vert ^{1/2}, \quad |\dot{\xi^\ep}(t)|\leq M \ep^{-3\gamma
_2/2}\vert\log \ep\vert ^{1/2}.
\end{equation}
\end{prop}

This is an analogue of \eqref{estimates} and will be fundamental
for our analysis of the emergence of interface.


\subsection{Deterministic and stochastic Allen--Cahn equations}\label{ss:deterministic-stochastic}

The (deterministic) Allen--Cahn equation was proposed in \cite{a-c}
as a model for the dynamics of interfaces in crystal structures
in alloys.  The same equation also appears as a model for various
other problems, including population genetics and nerve conduction.

As far as the one-dimensional case is concerned, the behavior of
the solution as $\ep\to 0$ was analyzed in \cite{CP, chen}. After
a very short time,  the value of the solution becomes close to
$a_+$ or $a_-$ in most part of the domain, thus generating
possibly many very steep transition layers. These well developed
transition layers then start to move very slowly, and each time a
pair of transition layers meet, the two layers annihilate each
other, thus the number of layers decrease gradually. Although
those collision-annihilation process takes place rather quickly,
the motion of layers between the collisions is extremely slow,
and the profile of the layers look nearly unchanged during those
slow motion periods; in other words, the solution exhibits a
metastable pattern. The situation is quite different in the
multi-dimensional case, where such metastable patterns hardly
appear because of the curvature effect on the motion of the
interface. This curvature effect  in higher dimensions is well
illustrated by the sharp interface limit $\ep\to 0$, where the
motion of layers (sharp interfaces) is known to be governed by
the mean curvature flow plus some driving forces. There is a
large  literature on the rigorous justification of this singular
limit; we refer, among others, to \cite{BS}, \cite{C1,C2},
\cite{MS1,MS2}, \cite{ahm}

Stochastic systems of Allen--Cahn type have been analyzed in
\cite{faris}. For the one-dimensional case, in  \cite{Fun95}, \cite{bmp},
the authors studied the stochastic Allen--Cahn equation with
initial data close to a Heaviside function. They proved under an
appropriate scaling that the solution stays close to this shape, while the random perturbation creates a dynamic for the
single interface which is observed on a much faster time
scale than in the deterministic case. This has been also studied
 in \cite{weber2} via an invariant measure approach.
 The author therein, under certain assumptions,
proves exponential convergence towards a curve of minimizers of
the energy, and a concentration of the measure on configurations
with precisely one jump. In \cite{OWW14}, the authors
studied the competition between some energy functional that is
minimized for small noise strength, and they also investigate the entropy induced by a system of large size.

If the initial data involves more than one interfaces, it is believed
that these interfaces also exhibit a random movement which is much
quicker than in the deterministic case, while different interfaces
should annihilate when they meet \cite{fatkullin}, and the
limiting process is related to the Brownian (see \cite{fontes}
for formal arguments).

As far as the sharp interface limit of the stochastic Allen--Cahn
equation \eqref{AC} is concerned, we first mention the pioneering
work of Funaki \cite{Fun99}: the law of motion of the limit
interface is rigorously derived, for dimension $n=2$ and
convex initial interface $\Gamma _0$ and it is given by
\begin{equation}\label{mmc}
V=(n-1)\kappa+c \dot W_t,
\end{equation}
where V is the inward normal velocity of the inner interface
$\Gamma_t$, $\kappa$ is the mean curvature of $\Gamma_t$, $\dot W_t$ is
the white noise in $t$ (namely the singular limit of the mild
noise as $\ep \to 0$)  and $c$ denotes an identified constant.
Note that this motion law was derived under the assumption that
the initial data is well-prepared, in the sense that it depends on
$\ep$ in such a way that it is very close to the formal
asymptotics, i.e., $U_0\left(\frac{d(x,0)}{\ep}\right)$, where
$U_0(z)$ is the underlying one-dimensional travelling wave and
$d(\cdot,0)$ the signed distance function to $\Gamma _0$
\footnote{In this paper, a signed distance function $d(\cdot,t)$
to $\Gamma _t$ is always negative in the region enclosed by
$\Gamma _t$, and positive elsewhere.}.
Later, in \cite{weber1}, the classical result of \cite{Fun99}
was extended  to spatial dimensions greater than two without the
restriction of initial convexity.

The multi-dimensional stochastic Allen--Cahn
equation driven by a multiplicative noise is studied in \cite{rogweb}.
This noise is non-smooth in time and smooth in space (finite sum of time-dependent
Brownian motions, with coefficients deterministic functions of the
spatial variables).
The authors prove for $\ep$-dependent initial data, the tightness of solutions for
the sharp interface limit problem and show convergence to
phase-indicator functions. The existence and properties of such a
stochastic flow,
was first established in \cite{yip98}, in the context of
geometric measure theory. More precisely, in \cite{yip98} an
iterative scheme is constructed, and a sequence of sets with
randomly perturbed boundaries is introduced. The analysis in
\cite{rogweb} was based on energy estimates and is related to the
construction of  \cite{yip98}. In \cite{bebp}, a stochastic
Allen--Cahn equation is considered; the authors study its
large deviation asymptotics in a joint sharp interface and small
noise limit.

The space-time white noise driven Allen--Cahn equation is known to
be ill-posed in space dimensions greater than one, \cite{Wal86},
\cite{DPZ92}. Therefore, in \cite{hiweb}, a multi-dimensional
stochastic Allen--Cahn equation with mollified additive white
space-time noise is analyzed (finite sum of time-dependent Brownian
motions with finite noise strength). For regular $\ep$-independent
initial data, it is shown that as the mollifier is removed, the
solutions converge weakly to zero, independently of the initial
condition. If the noise strength converges to zero at a
sufficiently fast rate, then the solutions converge to those of
the deterministic equation.
A large deviation principle is discussed in \cite{HaWe:14}.

Considering stochastic models where the Allen--Cahn operator
appears, or stochastic sharp interface limit problems from phase
separation, we also refer to the recent results of
\cite{ABBK}, \cite{AKM}, \cite{AFK}, \cite{ABKs}. More specifically, in \cite{ABBK}, the
mass conserving Allen--Cahn equation with noise was
analyzed and the stochastic dynamics of a droplet's motion along
the boundary in dimension $2$ were derived. In \cite{AKM},
the authors established the stochastic existence and investigated
the regularity of solutions for the so-called
Cahn-Hilliard/Allen--Cahn equation with space-time white noise;
for the same problem, in \cite{AFK}, Malliavin calculus was
applied for the proof of existence of a density for the
stochastic solution, in dimension one. The stochastic sharp
interface limit for the Cahn-Hilliard equation with space-time
smooth in space noise has been presented in \cite{ABKs}.

\subsection{Motivation for the current work}
Our work stands in the framework of \cite{Fun99} and
\cite{weber1}, where the mild noises (MN1) and (MN2), of
subsection \ref{ss:assumptions}, have been initially introduced.
As mentioned before, in these works, it is shown that the sharp
interface limit of \eqref{AC} is motion driven by mean curvature
plus an additional stochastic forcing term. This holds true for
well-prepared initial data. However, whether or not this motion
law is valid for a large class of initial data has never been
studied. To answer this question, one has to study the {\it
generation of interface} in details. We analyze first the
solution's profile for short times, and show that layers of
thickness $\mathcal O(\ep)$ are rapidly formed. Then, considering
later times, we prove that for rather general initial data the
thickness of the layers remains of order $\mathcal O(\ep)$, and we
determine the shape of  the solution $u^\ep(x,t)$ {\it inside} the layers.

To do so, we shall rely on the results of \cite{ahm} and \cite{am}
for the deterministic Allen--Cahn equation
\begin{equation}\label{AC1}
\partial _t u= \Delta u + \frac{1}{\ep^2} \left(f(u)-\ep g^\ep(x,t)\right).
\end{equation}
The authors in \cite{ahm} showed that, for a rather general class
of initial data that are independent of $\ep$, the solution
$u^\ep(x,t)$ of \eqref{AC1} develops a steep internal layer within
a short time interval of $\mathcal O(\ep^2 \ln \ep)$.
Consequently, $u^{\ep}(x,t)$ lies between a pair of super- and
sub-solutions $u^+$, $u^-$ for $t_{\ep} \le t \le T$, whose
profiles are very close to the formal asymptotics of typical
fronts and are located within the distance of $\mathcal O(\ep)$
from each other. Since the fronts of both $u^+$, $u^-$ move by the
correct motion law with an error margin of $\mathcal O(\ep)$, so
does the front of $u^{\ep}$. This indicates that the layers of
$u^\ep(x,t)$ move by motion by mean curvature plus an additional
pressure term, and that the their thickness is $\mathcal O (\ep)$.
Recently, the authors in \cite{am} have found a way to explore the
profile of the solution $u^{\ep}(x,t)$ {\it inside} these layers. More
precisely, they have proved the validity of the principal term of
the formal asymptotic expansions for rather general initial data.

Our present analysis of the singular limit of problem \eqref{AC}
reveals, in particular, that the profile of the solutions
$u^\ep(x,t)$ is not altered by the mild noise, for both the
thickness of the layers (compare Theorem \ref{th:main} with
\cite{ahm}) and the profile {\it inside} the layers (compare Theorem
\ref{th:profile} with \cite{am}). The main difference with the
deterministic problem stands in a slight shift of the position of
the layers which occurs in the very early times. This will be
clarified in Section \ref{s:emergence}.

Let us underline that, while the perturbation term $g^\ep(x,t)$ in
\eqref{AC1} remains uniformly bounded as $\ep \to 0$ \cite{ahm},
in the present paper, we allow the perturbation $\xi^\ep(t)$ to
become singular as $\ep \to 0$, as can be seen in
\eqref{estimates} and \eqref{estimates2}. We therefore need to
modify our argument for the generation of interface (see Section
\ref{s:emergence}) and then to use the stochastic approach of
\cite{Fun99}, \cite{DLN}, \cite{weber1}. The latter is suitable
for perturbations $\xi^\ep(t)$ which behaves like white noise as
$\ep\rightarrow 0$, resulting to random dynamics in the limit, in
contrast with \cite{ahm}.

\section{On stochastic motion by mean curvature}\label{s:mmc}

Before stating our main results, we need to give a precise
definition of the motion law of the form \eqref{mmc} for the limit interface.
The interpretation of this motion law actually depends on the type
of noise under consideration, namely the (MN1) type noise and the
(MN2) type one mentioned earlier.


\subsection{Motion law for the (MN1) type noise}\label{ss:mmc-fun}

The interpretation of the motion law \eqref{mmc} for this type of noise
was clarified by Funaki \cite{Fun99}.  In this subsection we will adopt
his definition and first recall some of his results.  Note that
the interpretation of \eqref{mmc} in this sense holds
as long as
the random curve $\Gamma _t$ remains strictly convex and does
not touch the boundary $\partial \Omega$.


Let $c_0>0$ and $\alpha _0>0$ be  given constants (which will be
taken as in \eqref{def:c0} and \eqref{def:alphazero} in our
context). A strictly convex curve $\Gamma$ can be parametrized by
$\theta \in [0,2\pi)$ in terms of the Gauss map: the position $x$
on $\Gamma$  is denoted by $x(\theta)$ if the angle between a
fixed direction and the outward normal $n(x)$ at $x$ to $\Gamma$
is $\theta$. Denote by $\kappa=\kappa(\theta)$ the (mean) curvature of
$\Gamma$ at $x=x(\theta)$. Then the stochastic motion by mean
curvature dynamics
$$
V=\kappa +c_0\alpha _0 \dot W _t
$$
is defined through the nonlinear stochastic partial differential
equation for $\kappa=\kappa(\theta,t)$:
\begin{equation}\label{eq-kappa}
\partial _t \kappa=\kappa ^2 \partial _{\theta\theta}\kappa +
\kappa ^3+c_0\alpha _0\kappa ^2 \circ \dot W_t, \quad 0<t<
\sigma, \theta \in [0,2\pi),
\end{equation}
where $\circ$ means the Stratonovich stochastic integral and $\sigma=\lim _{N\to \infty} \sigma _N$.
Stopping times are defined by
$$
\sigma _N:=\inf \{t>0,\, \bar  \kappa _t >N \text{ or } \text{dist}(\Gamma _t,\partial \Omega)<1/N\}, \quad N> 0,
$$
where $\bar \kappa _t=\max_{\theta \in [0,2\pi)}\max\{\kappa(\theta,t),\kappa ^{-1}(\theta,t), \vert \partial _\theta \kappa (\theta,t)\vert\}$.
Indeed, once the mean curvature $\kappa (\theta,t)$ is obtained via \eqref{eq-kappa}, one can determine $\Gamma _t=\{x_t(\theta)\in \R^2\cong
\mathbb C, \, \theta \in [0,2\pi)\}$ by formula \cite[(1.10)]{Fun99} to which we refer for further details.

Also, we need to consider approximations of this motion as
follows. Let $\gamma _0^\ep$ (which will be taken as in
\eqref{def:shifted-initial-interface} in our context) be a
$C^{2,\alpha}$ hypersurface which is a slight shift of $\Gamma
_0$, in the sense that
$$
\gamma _0^\ep \to \Gamma _0  \text{ as } \ep \to 0, \text{ in the
$C^{2,\alpha}$ sense.}
$$
 Furthermore, we replace the forcing term $c_0 \alpha _0 \dot W_t$ by
$-\frac{c(\ep\xi^\ep(t))}{\ep}$, where $\delta\mapsto c(\delta)$
(which will be taken as in \eqref{n1n} in our context) is smooth
in a neighborhood of zero, satisfies $c(0)=0$ and $\partial _\delta c(0)=-c_0$. We
are therefore equipped with a family of hypersurfaces $(\gamma _t
^\ep)_{0\leq t<\sigma  ^{\ep}}$, starting from $\gamma _0 ^\ep$ and
evolving with the law
$$
V=\kappa -\frac{c(\ep \xi ^\ep(t))}{\ep} \quad\text{ on } \gamma
^\ep _t.
$$
Here we have $\sigma ^{\ep}=\lim _{N\to \infty} \sigma _N^{\ep}$, where
\begin{equation}
\label{analogous}
\sigma _N^{\ep}:=\inf \{t>0,\quad \bar  \kappa _t ^{\ep} >N
\text{ or } \text{dist}(\gamma ^{\ep} _t,\partial \Omega)<1/N\}, \quad N> 0,
\end{equation}
where
\begin{equation}
\label{def:kappabar}
\bar \kappa _t^{\ep}=\max_{\theta \in [0,2\pi)}\max\{\kappa^{\ep}
(\theta,t),(\kappa^{\ep}) ^{-1}(\theta,t), \vert \partial _\theta
\kappa ^{\ep}(\theta,t)\vert\},
\end{equation}
 with $\kappa ^{\ep}$ the mean curvature of $\gamma ^{\ep}_t$.

{} Since $-\frac{c(\ep\xi^\ep(t))}{\ep}\sim c_0\xi ^\ep(t)$ as $\ep \to 0$, and since $\xi  ^{\ep}(t)$
converges to $\alpha _0 \dot W_t$ in distribution sense, it is expected that the approximations
$(\gamma _t ^\ep)$ converge, in some sense, to $(\Gamma _t)$.
Using the martingale method such a convergence---see Corollary \ref{cor:funaki}
for a precise statement---is proved in \cite{Fun99}, when $\gamma _0 ^{\ep}=\Gamma _0$.
In particular, for all but countable many $N>0$ we have $\sigma _N^\ep \to \sigma _N$ as $\ep \to 0$.


\subsection{Motion law for the  (MN2) type noise}\label{ss:mmc-web}

The precise meaning of the motion law \eqref{mmc} in the context of the (MN2) type noise can be clarified by using the results of Dirr, Luckhaus and Novaga \cite{DLN}. In this subsection we summarize their results and apply them to our problem. More precisely, we refer to
\cite[Theorem 3.1]{DLN} for the existence result and to
\cite[Corollary 4.2]{DLN} for the estimate of the deviation from
the original problem, when the white noise is smoothly
approximated and when the initial hypersurface is slightly
shifted.

\medskip

Let $c_0>0$ be a given constant (which will be taken as in
\eqref{def:c0} in our context). Since the initial hypersurface
$\Gamma _0=\partial \Omega _0$ is of class $C^{2,\alpha}$, there
is a stopping time $\tau=\tau(\Gamma _0)=\tau(\omega,\Gamma _0)$
depending on the $C^{2,\alpha}$-norm of $\Gamma _0$, and a family
of hypersurfaces $(\Gamma _t)_{0\leq
t<\tau}=(\Gamma_t(\omega))_{0\leq t<\tau(\omega,\Gamma _0)}$ of
class $C^{2,\alpha}$, such that, for any $X_0\in \Gamma _0$, there
is a process $X(\cdot)$ with $X(t)=X(t,\omega) \in \Gamma
_t=\Gamma _t(\omega)$ for almost all $\omega\in \Omega_{prob}$
which solves the It\^o equation
$$
dX=\nu(X(t,\omega),t)(n-1)\kappa(X(t,\omega),t)dt+\nu(X(t,\omega),t)c_0dW,
\quad X(0)=X_0,
$$
where $\kappa(y,t)$ and $\nu(y,t)$ are respectively the mean
curvature and the inner normal at $y\in \Gamma _t$. This is the
sense we adopt for the motion law $$
V=(n-1)\kappa+c_0 \dot W_t,
$$
or
$dV=(n-1)\kappa dt+c_0 dW_t$, which we call the stochastic motion by mean curvature.

Also, we need to consider approximations of this motion as
follows. Let $\gamma _0^\ep$ (which will be taken as in
\eqref{def:shifted-initial-interface} in our context) be a
$C^{2,\alpha}$ hypersurface which is a slight shift of $\Gamma
_0$, in the sense that
$$
\gamma _0^\ep \to \Gamma _0  \text{ as } \ep \to 0, \text{ in the
$C^{2,\alpha}$ sense.}
$$
Furthermore, we replace the forcing term $c_0 \dot W_t$ by
$-\frac{c(\ep\xi^\ep(t))}{\ep}$, where $\delta\mapsto c(\delta)$
(which will be taken as in \eqref{n1n} in our context) is smooth
in a neighborhood of zero, satisfies $c(0)=0$ and $\partial
_\delta c(0)=-c_0$. We are therefore equipped with a family of
hypersurfaces $(\gamma _t ^\ep)_{0\leq t<\tau  ^{\ep}}$, starting
from $\gamma _0 ^\ep$ and evolving with the law
$$
V=(n-1)\kappa -\frac{c(\ep \xi ^\ep(t))}{\ep} \quad\text{ on } \gamma
^\ep _t.
$$
{}From the definition of the noise $\xi^\ep(t)$ as the derivative
of an approximated Brownian motion $W^\ep(t)$ (by convolution with
a mollifier) and the above assumptions, we have---see \cite[Lemma
3.3] {weber1}---that, for any $T>0$, the random
functions $t\mapsto \int _0^t -\frac{c(\ep\xi^\ep(s))}{\ep}ds$
converge almost surely to $t\mapsto c_0W(t)$ in
$C^{0,\alpha}([0,T])$ for any $0<\alpha<\frac 12$. This enables to
quote \cite[Corollary 4.2] {DLN}: there is a time $T>0$ such that
\begin{equation}
\label{cor-dirr} \sup _{0\leq t\leq T} \Vert d(t,x)-d^{\ep}(t,x)
\Vert _{C ^{2,\alpha}} \to 0, \quad\text{ as } \ep \to 0,
\end{equation}
where $d(\cdot,t)$, $d^{\ep}(\cdot,t)$ denote the signed distance
functions to $\Gamma _t$, $\gamma _t ^{\ep}$ respectively.

\section{Main results}\label{s:main}

Our first main result is to localize the transitions layers of the
solution of the stochastic Allen--Cahn equation in a $\mathcal
O(\ep)$ neighborhood of a family of hypersurfaces $(\tildet)$,
which is defined as follows. The initial hypersurface $\tildezero$
is defined in \eqref{def:shifted-initial-interface} and is a
slight shift of the initial interface $\Gamma _0$ defined in
\eqref{def:initial-interface} (we hope that the reason for such a
shift will become transparent for the reader in Section
\ref{s:emergence}).

Let the family $(\tildet)$ evolve with the law of motion
\begin{equation}\label{law-motion}
V=(n-1)\kappa-\frac{c(\ep \xi ^\ep (t))}\ep \quad \text{ on } \tildet,
\end{equation}
where $c(\delta)$ is the speed of the bistable traveling wave
$m(z;\delta)$ defined in \eqref{n1n}. Recalling Section \ref{s:mmc}, if the noise is of the
(MN1) type then this family is defined for $0<t\leq \sigma _N^\ep$, $N>0$ arbitrary,
whereas if the noise is of the (MN2) type this family is defined for $0<t\leq \tau^\ep$.
In the latter case, let $T>0$ be given as in \eqref{cor-dirr}. Also $\tildeint$ denotes the
region enclosed by $\tildet$.

\begin{theorem}[Emergence and motion of $\mathcal O(\ep)$ Allen--Cahn
layers]\label{th:main} Let the nonlinearity $f$ and the initial
data $u_0$ satisfy the assumptions of subsection
\ref{ss:assumptions}, and the mild noise be of (MN1) or (MN2)
type. In the former case, let $N>0$ be given. Let $u^\ep(x,t)$ be
the solution of \eqref{AC}. Let $\eta \in (0,\eta _0:= \min
(a-a_-,a_+ -a))$ be arbitrary and define $\mu$ as the derivative
of $f(u)$ at the unstable zero $u=a$, that is
\begin{equation}\label{def:mu}
\mu=f'(a)>0.
\end{equation}

Then there exist positive constants $\ep _0 $ and $C$ such that,
for all $\,\ep \in (0,\ep _0)$ and for all $t^\ep\leq t \leq \sigma ^\ep _N$---if noise is of
(MN1) type---or all $t^\ep \leq t \leq
T$---if noise is of (MN2) type---where
$$
t^\ep:=\mu _\ep  ^{-1} \ep ^2 |\ln \ep|, \text{ with } \mu_\ep \to
\mu \text{ as } \ep \to 0,
$$
 we have
\begin{equation}\label{resultat}
u^\ep(x,t) \in
\begin{cases}
\,[a_--\eta,a_++\eta]\quad\text{if}\quad
x\in\mathcal N_{C\ep}(\tildet)\\
\,[a_--\eta,a_-+\eta]\quad\text{if}\quad
x\in\tildeint\setminus\mathcal N_{C\ep}(\tildet)\\
\,[a_+-\eta,a_++\eta]\quad\text{if}\quad
x\in(\tildeext)\setminus\mathcal N_{C\ep}(\tildet),
\end{cases}
\end{equation}
where $\mathcal N _r(\tildet):=\{x\in \Omega:\, dist(x,\Gamma
_t)<r\}$ denotes the $r$-neighborhood of $\tildet$.
\end{theorem}

As mentioned in the introduction and clear from the above theorem,
the deterministic $\mathcal O(\ep)$ thickness of the layers of the
solutions $u^\ep(x,t)$---as estimated in \cite{ahm}---is not
altered by the mild noise.

The above theorem enables to generalize the convergence results of
\cite{Fun99} and \cite{weber1}---which are concerned with
well-prepared initial data---to rather general data.

\begin{corollary}[Extension of Funaki \cite{Fun99} to general data]
\label{cor:funaki}
Let the nonlinearity $f$ and the initial data $u_0$ satisfy the
assumptions of subsection \ref{ss:assumptions}. Let the mild
noise be of (MN1) type. Assume further that $\xi ^\ep(0)=0$. Let
$u^\ep(x,t)$ be the solution of \eqref{AC}. Assume further that
$n=2$ and that $\Omega _0$ is convex. Following subsection
\ref{ss:mmc-fun}, let
$(\Gamma _t)_{0\leq t<\sigma:=\lim _{N\to\infty} \sigma _N}$
evolve by
$$
V=\kappa+(c_0 \alpha _0) \dot W_t,
$$
with $c_0>0$ the constant defined in \eqref{def:c0}, and
\begin{equation}
\label{def:alphazero}
\alpha_0:=\sqrt{2\int _0^{\infty}E[\xi _0 \xi_t]dt}.
\end{equation}
 Then the  random motion of curves $(\gamma _t
 ^\ep)_{0\leq t <\sigma ^\ep:=\lim _{N\to \infty}\sigma _N ^\ep}$
 defined in subsection \ref{ss:mmc-fun}
 satisfies the following two conditions.
 \begin{itemize}
 \item [$(i)$] Let $N>0$ be given. For $0\leq t
<\sigma ^\ep _N$, let $x\mapsto \Phi^\ep(x,t)$ be the step
function with value $a_-$ in the region enclosed by $\gamma ^\ep
_t$ and $a_+$ elsewhere. Then
$$
\sup _{t^\ep \leq t \leq \sigma ^\ep _N} \Vert u^\ep(\cdot,
t)-\Phi^\ep(\cdot,t)\Vert _{L^2(\Omega)} \to 0 \quad \text{ in
probability, as } \ep \to 0,
$$
where $t^\ep$ is as in Theorem \ref{th:main}.

 \item[$(ii)$] $\gamma _t ^\ep$ converges to $\Gamma _t$ as $\ep
 \to 0$ in the following sense: for any $T>0$ and all but
 countable many $N\in \R ^+$, the joint distribution of $(\sigma
 _N ^\ep,\Gamma ^\ep _{t\wedge\sigma ^\ep _N})$ on $\R ^+ \times
 C([0,T],C([0,2\pi),\R ^2))$ converges, as $\ep \to 0$, to that of
 $(\sigma _N,\Gamma _{t\wedge \sigma _N})$.
 \end{itemize}
\end{corollary}

\begin{corollary} [Extension of Weber \cite{weber1} to general data]\label{cor:weber}
Let the nonlinearity $f$ and the initial
data $u_0$ satisfy the assumptions of subsection
\ref{ss:assumptions}. Let the mild noise be of (MN2) type. Let
$u^\ep(x,t)$ be the solution of \eqref{AC}. Following subsection
\ref{ss:mmc-web}, let $(\Gamma _t)_{0\leq t < \tau(\Gamma _0)}$
evolve by
$$
dV=(n-1)\kappa dt +c_0 dW_t,
$$
with $c_0>0$ the constant defined in \eqref{def:c0}. For $0\leq t
<\tau(\Gamma _0)$, let $x\mapsto \Phi(x,t)$ be the step function
with value $a_-$ in the region enclosed by $\Gamma _t$ and $a_+$
elsewhere. Let $T>0$ be as in \eqref{cor-dirr}.

Then
$$
\sup _{t^\ep \leq t \leq T} \Vert u^\ep(\cdot,
t)-\Phi(\cdot,t)\Vert _{L^2(\Omega)} \to 0 \quad \text{ almost
surely, as } \ep \to 0,
$$
where $t^\ep$ is as in Theorem \ref{th:main}.
\end{corollary}

By performing formal asymptotic expansions, see \cite[Section
2]{ahm} for more details, it is suspected that, close to the limit
interface $\Gamma _t$, the solution is approximated by
\begin{equation}\label{ansatz}
u^\ep(x,t)\sim U_0\left(\frac{d(x,t)}{\ep}\right)+\cdots,
\end{equation}
with $d(\cdot,t)$  the signed distance function to $\Gamma _t$,
and $U_0(z)$  the unique solution (whose existence is guaranteed
by the integral condition \eqref{equalwell}) of the stationary
problem
\begin{equation}\label{eq-phi}
 \left\{\begin{array}{ll}
 {U_0} '' +f(U_0)=0 \vspace{2pt}\\
 U_0(-\infty)= a_-\,,\quad U_0(0)=a\,,\quad U_0(\infty)= a_+\,.
 \end{array} \right.
\end{equation}
This represents the first approximation of the profile of a
transition layer around the interface observed in the stretched
coordinates. As recently proved in \cite{am} for the deterministic
case, we are actually able to prove the validity of the first term
of expansion \eqref{ansatz} for the stochastic case under
consideration.

Let us define the level surface of the solution $u^\ep$
\begin{equation}\label{level-sets}
\Gamma _t ^\ep:=\{x\in \Omega:\,u^\ep(x,t)=a\}
\end{equation}
and the {\it signed distance function associated with
$\Gamma^\ep$} by
\begin{equation}\label{eq:dist-e}
 \overline{d^\ep} (x,t):=
 \begin{cases}
 -&\hspace{-10pt} \mbox{dist}(x,\Gamma^\ep _t) \quad \text{if }
 u^\ep(x,t)<a \\
 &\hspace{-10pt} \mbox{dist}(x,\Gamma^\ep _t)\quad\text{if }
 u^\ep(x,t)>a\,.
 \end{cases}
\end{equation}
Recall that if noise is of the (MN2) type, then $T>0$ was defined in
\eqref{cor-dirr}. To unify the notation in the following, if noise is of the (MN1) type, then,
for a given $N>0$, we select $0<T<\sigma _N$ (see subsection \ref{ss:mmc-fun}). It therefore
follows from \eqref{resultat} that $\Gamma _t ^\ep \subset
\mathcal N _{C\ep}(\gamma _t ^\ep)$ for all $t^\ep \leq t \leq T$,
so that
\begin{equation}\label{dist-ep}
\vert  \overline{d^\ep} (x,t)-d^\ep(x,t)\vert \leq C \ep \quad
\forall (x,t)\in \overline \Omega\times [t^\ep,T], 0<\ep<<1.
\end{equation}

\begin{theorem}[Profile in the layers]\label{th:profile} Let the
assumptions of Theorem \ref{th:main} hold. Fix $\rho >1$ and
$0<T'<T$. Then
\begin{enumerate}
\item [{\rm (i)}] If $\ep >0$ is small enough then, for any
$t\in[\rho t^\ep,T']$, the level set $\Gamma _t ^\ep$ is a smooth
hypersurface and can be expressed as a graph over $\gamma _t
^\ep$. \item[{\rm (ii)}] We have
\begin{equation}\label{lim-d-epsilon}
\lim_{\ep \to 0}\; \sup_{\rho t^\ep \leq t\leq T',\, x\in \bar
\Omega} \left |u^\ep(x,t)-U_0\left(\frac
{\overline{d^\ep}(x,t)}{\ep}\right)\right|=0\,,
\end{equation}
where $\overline{d^\ep}$ denotes the signed distance function
associated with $\Gamma^\ep$.
\end{enumerate}
\end{theorem}

As mentioned in the introduction and clear from the above theorem,
the deterministic profile of the solutions $u^\ep(x,t)$ {\it inside} the
layers---as explored in \cite{am}---is not altered by the mild
noise.

\medskip

The rest of the paper is organized as follows. In Section
\ref{s:emergence}, we prove the emergence of internal layers for
the problem \eqref{AC}. In Section \ref{s:motion}, we construct
accurate sub- and super-solutions to study the motion of the
layers that then takes place. The combination of these two studies
is performed in Section \ref{s:layers} where we prove Theorem
\ref{th:main} which, using the results of \cite{Fun99},
\cite{DLN}, \cite{weber1}, implies Corollaries \ref{cor:funaki}
and \ref{cor:weber}. Last, we prove Theorem
\ref{th:profile} in Section \ref{s:profile}.


\section{Rapid emergence of $\mathcal O (\ep)$ layers}
\label{s:emergence}

This section deals with the emergence of internal layers (or the
generation of interface) which occurs very quickly. In other
words, given a virtually arbitrary initial data, we prove that the
solution $u^\ep(x,t)$ quickly becomes close to $a_\pm$ in most
part of $\Omega$. In order to track the $\mathcal O(\ep)$
thickness of the layers, we show that the generation occurs in a
$\mathcal O(\ep)$ neighborhood of some smooth hypersurface
$\tildezero$, which itself lies in a $o(\ep ^{1-\gamma})$
neighborhood of the initial interface $\Gamma _0$, where $\gamma$
is chosen such that
$$
\begin{cases} 0<\gamma _1<\gamma<\frac 13 \mbox{ if the noise is of the (MN1) type}\\
0<\frac{\gamma _2}{2}<\gamma<\frac 13\mbox{ if the noise is of the
(MN2) type}.
\end{cases}
$$

 The reason for such an initial drift is the following. For $0\leq t\leq T$, the mean value theorem
 provides a $0<\theta<1$ such that
\begin{equation}\label{crucial}
\xi^\ep(t)=\xi^\ep(0)+\dot \xi^\ep(\theta t)t=\xi^\ep(0)+o(\ep),
\end{equation}
as long as $0\leq t \leq \mathcal O(\ep ^2|\ln \ep|)$, where we
have used \eqref{estimates} under the noise
assumption (MN1), and \eqref{estimates2} under
the noise assumption (MN2). Once the crucial observation
\eqref{crucial} is made, the treatment of the $o(\ep)$ term
follows from the generation of interface property performed in
\cite[Section 4]{ahm}, whereas the $\xi ^\ep(0)=o(\ep ^{-\gamma})$
term explains the initial shift.
\medskip

In order to take advantage of observation \eqref{crucial}, we
define
\begin{equation}\label{def:fep}
f^\ep(u):=f(u)+\ep \xi ^\ep(0).
\end{equation}
In view of assumptions \eqref{cf0} and \eqref{cf1} on $f$, and
since $\ep \xi ^\ep(0)=o (\ep ^{1-\gamma}) \to 0$, we have, for
$\ep >0$ small enough, that $f^\ep$ is still of the bistable type,
in the sense that
\begin{equation}\label{cf0ep}
f^\ep \mbox{ has exactly three zeros} \ \ a_- ^\ep<a^\ep<a_+^\ep,
\end{equation}
where $a_- ^\ep=a_-+o(\ep ^{1-\gamma})$, $a^\ep=a+o(\ep
^{1-\gamma})$, $a_+^\ep=a_++o(\ep ^{1-\gamma})$, and
\begin{equation}\label{cf1ep}
\frac d{du}f^\ep(a_\pm^\ep)\to f'(a_\pm)<0,\ \ \mu_\ep:=\frac
d{du}f^\ep(a^\ep)\to \mu=f'(a)>0.
\end{equation}
We now define
\begin{equation}\label{def:shifted-initial-interface}
\tildezero:=\left\{x\in\Omega:\;u_0(x)=a^\ep\right\},
\end{equation}
which consists in a $o(\ep ^{1-\gamma})$ shift of the initial
interface $\Gamma _0$ defined in \eqref{def:initial-interface}. In
view of assumptions in subsection \ref{ss:assumptions},
$\tildezero$ is a smooth hypersurface without boundary and
properties analogous to \eqref{hyp:gradient} and
\eqref{def:interior} hold true with obvious changes. In
particular, thanks to the compactness of $\Gamma _0$,
\eqref{hyp:gradient} is transferred into
\begin{equation}\label{gradientep}
\nabla u_0(x)\cdot n^\ep(x)\geq \vartheta>0\; \text{ for any }
x\in\tildezero,
\end{equation}
for all $\ep >0$ small enough. We can now state our generation of
interface result.

\begin{theorem}[Emergence of $\mathcal O(\ep)$ layers around
$\tildezero$]\label{th:emergence}
Let the nonlinearity $f$ and the
initial data $u_0$ satisfy the assumptions of subsection
\ref{ss:assumptions}. Let the mild noise be of (MN1) or (MN2)
type. Let $u^\ep(x,t)$ be the solution of \eqref{AC}. Let $\eta
\in (0,\eta _0:= \min (a-a_-,a_+ -a))$ be arbitrary.

Then there exist positive constants $\ep_0$ and $M_0$ such that,
for all $\,\ep \in (0,\ep _0)$,
\begin{enumerate}
\item [$(i)$] for all $x\in\Omega$,
\begin{equation}\label{g-part1}
a_- -\eta \leq u^\ep(x,\mu _\ep ^{-1} \ep ^2 | \ln \ep |) \leq a_+
+\eta,
\end{equation}
\item [$(ii)$] for all $x\in\Omega$ such that $|u_0(x)-a^\ep|\geq
M_0 \ep$, we have that
\begin{align}
&\text{if}\;~~u_0(x)\geq
a^\ep+M_0\ep\;~~\text{then}\;~~u^\ep(x,\mu _\ep ^{-1} \ep ^2 | \ln
\ep |)
\geq a_+ -\eta,\label{g-part2}\\
&\text{if}\;~~u_0(x)\leq a^\ep
-M_0\ep\;~~\text{then}\;~~u^\ep(x,\mu _\ep  ^{-1} \ep ^2 | \ln \ep
|)\leq a_-+\eta \label{g-part3}.
\end{align}
\end{enumerate}
\end{theorem}

\begin{proof} In view of the crucial observation \eqref{crucial} and definition \eqref{def:fep}, the Allen--Cahn equation \eqref{AC} is recast (for small enough
times)
$$
\partial _t u= \Delta u + \frac{1}{\ep^2}\left(f^\ep(u)-\ep
g^\ep(t)\right),\;\;\;0<t\leq \mu _\ep ^{-1}\ep ^2|\ln \ep|,\;\;\;x
\in\Omega,
$$
where the perturbation term
$$
g^\ep(t):=-\xi ^\ep(t)+\xi ^\ep (0),
$$
satisfies
$
\Vert g^\ep\Vert _{L^\infty(0, \mu _\ep ^{-1}\ep ^2|\ln \ep|)}=o(\ep)$, as $\ep \to 0
$. Moreover, using \eqref{estimates} under the noise
assumption (MN1), and \eqref{estimates2} under
the noise assumption (MN2), we get that, in any case,
$$
\Vert \dot g^\ep\Vert _{L^\infty(0, \mu _\ep ^{-1}\ep ^2|\ln \ep|)}=\mathcal O\left(\ep ^{-1} \right), \; \text{ as } \ep \to 0.
$$
After writing the problem in such a form and as far as the perturbation term is concerned,  we are in the footsteps of the Allen--Cahn
equation $(P^{\ep})$ studied in
\cite{ahm}, since the above estimate corresponds to assumption (1.3) in \cite{ahm}.

On the other hand, we need to handle the following minor change: $f$ in \cite{ahm} is replaced by $f^\ep$
in our setting. This difference implies that
 $a$ in \cite{ahm} is replaced by $a^\ep$ and is the reason why the generation occurs
around $\tildezero$ (and not around $\Gamma _0$). Nevertheless,
it is completely transparent that $f^{\ep}$ is still of the
bistable type {\it uniformly with respect to small $\ep >0$}
(this property is used in order to derive certain
estimates).
 More precisely, \eqref{cf0ep} and \eqref{cf1ep} correspond to
assumption (1.1) in \cite{ahm}, uniformly with respect to small
$\ep >0$. Similarly, the non degeneracy assumption
\eqref{gradientep}, when crossing the initial interface
$\tildezero$, is uniform with respect to small $\ep >0$ and
corresponds to assumption (1.10) in \cite{ahm}.

We can then construct the analogous of the sub- and supersolutions
of \cite[Section 4]{ahm}, namely
$$
w^\pm(x,t)=Y^{\ep}\left(\frac{t}{\ep^2}, u_0(x)\pm\ep^2C(e^{\tilde \mu _\ep \frac{t}{\ep ^2}}-1);\pm\ep\right),
$$
where $C>0$ is a large constant, $\tilde \mu _\ep$ is a very small perturbation of $\mu _\ep$, and $Y^{\ep}(\tau,\xi;\delta)$ is the solution of the
Cauchy problem
\begin{equation*}
\left\{\begin{array}{ll} Y_\tau ^{\ep}(\tau,\xi;\delta)=f^{\ep}
(Y^{\ep}(\tau,\xi;\delta))+\delta \quad \text { for }\tau >0\vspace{3pt}\\
Y^{\ep}(0,\xi;\delta)=\xi.
\end{array}\right.
\end{equation*}
Notice that, in this very early stage of emergence of the layers, the above sub- and supersolutions are obtained by considering only the
nonlinear reaction term, that is diffusion is neglected.

For the aforementioned reasons, we can reproduce the lengthy arguments of \cite[Section 4]{ahm} to  prove Theorem \ref{th:emergence},
which is nothing else that the analogous of  \cite[Theorem 3.1]{ahm} taking into account the change $f\leftarrow f^\ep$.
\end{proof}

\section{Propagation of $\mathcal O(\ep)$ layers}\label{s:motion}

In this section, we construct a pair of sub- and supersolutions
whose role is to capture in an $\mathcal O(\ep)$ sandwich the
layers of the solution $u^\ep(x,t)$, while they are propagating. In order to proceed to the aforementioned construction, we
need to define first properly some traveling waves and a signed
distance function used in the definition of this pair.

\subsection{Some traveling waves}\label{ss:preliminaries1}
For $\delta _0>0$ small enough and any $|\delta|\leq \delta _0$, the function $u\mapsto f(u)+\delta$ is still of bistable
type, and we denote by
$$
a_-(\delta)=a_-+\mathcal O(\delta)<a(\delta)=a+\mathcal
O(\delta)<a_+(\delta)=a_++\mathcal O(\delta),
$$
 its three zeros.

 Let
$c(\delta)$, $m(z;\delta)$ be the speed and the profile of the
unique traveling wave associated with the one dimensional problem
$$
\partial _t v=v_{zz}+f(v)+\delta,\;\;t>0,\;z\in \mathbb{R}.
$$
In other words, we have
\begin{equation}\label{n1n}
\begin{split}
&m_{zz}(z;\delta)+c(\delta)m_z(z;\delta)+f(m(z;\delta))+\delta=0,
\quad z\in \mathbb{R},\\
&m(-\infty;\delta)=a_-(\delta),\quad m(0;\delta)=a(\delta),\quad
m(+\infty;\delta)=a_+(\delta).
\end{split}
\end{equation}
Notice in particular that the assumption of balanced nonlinearity
\eqref{equalwell} implies $c(0)=0$. Moreover, the following
estimates are well-known (see in \cite{Che-Hil-Log},
\cite{Fun99} or \cite{weber1}).
\begin{lemma}[Estimates on traveling waves]\label{lem:tw-estimates}
There exist constants $\delta _0>0$, $C>0$, $\lambda>0$ such that,
for all $|\delta|\leq \delta _0$,
\begin{equation}\label{eql1}
\begin{split}
&0<a_+(\delta)-m(z;\delta)\leq Ce^{-\lambda|z|},\;\;\;z\geq 0,\\
&0<m(z;\delta)-a_-(\delta)\leq Ce^{-\lambda|z|},\;\;\;z\leq 0,\\
&0<m_z(z;\delta)\leq Ce^{-\lambda|z|},\;\;\;z\in\mathbb{R},\\
&|m_{zz}(z;\delta)|\leq Ce^{-\lambda|z|},\;\;\;z\in\mathbb{R},\\
&|m_\delta(z;\delta)|\leq C,\;\;\;z\in\mathbb{R},
\end{split}
\end{equation}
and
\begin{equation}\label{def:c0}
\partial _\delta c(0)=-c_0:=-\frac{a_+-a_-}{\displaystyle \int
_{a_-}^{a_+}\sqrt{2F(u)}\,du}<0,\quad F(u):=\int _u
^{a_+}f(z)\,dz.
\end{equation}
\end{lemma}


\subsection{Signed distance functions}\label{ss:preliminaries2}
We recall that the family of hypersurfaces $(\tildet)$ follows the law \eqref{law-motion} with initial data
$\tildezero$ defined in \eqref{def:shifted-initial-interface}. If the noise is of the (MN1) type
then it follows from \eqref{analogous} and \eqref{def:kappabar} that, up to reducing $\ep _0$ if
necessary,
$$
\mathcal K:=\sup_{0<\ep<\ep_0} \sup _{0\leq t\leq \sigma  ^{\ep}_N} \sup
_{y\in\gamma ^\ep _t} \sup _{1\leq i \leq n-1}\vert \kappa ^\ep
_i(y,t)\vert<\infty,
$$
with $\kappa ^\ep _i(y,t)$ the $i$-th principal curvature of $\gamma
_t^\ep$ at point $y$. On the other hand, if the noise is of the (MN2) type then it follows
from  \eqref{cor-dirr} that, up to reducing $\ep _0$ if
necessary,
$$
\mathcal K:=\sup_{0<\ep<\ep_0} \sup _{0\leq t\leq T} \sup
_{y\in\gamma ^\ep _t} \sup _{1\leq i \leq n-1}\vert \kappa ^\ep
_i(y,t)\vert<\infty.
$$
In the sequel we unify the notations by letting  $\mathcal T=\sigma _N^\ep$,  $\mathcal T=T$
if the noise is of the (MN1) type, (MN2) type respectively.

Let $\tildeint$ denote the region enclosed
by $\tildet$. We then define the associated signed distance
function by
\begin{equation}\label{signdist*}
  \tilde{d}^\ep(x,t):=\begin{cases}
    -{\rm dist}(x,\tildet) & \text{for   }x\in\tildeint, \\
    +{\rm dist}(x,\tildet) & \text{for  }x\in\tildeext.
  \end{cases}
\end{equation}
For $d_0>0$, choose an increasing function $\varphi\in
C^\infty(\mathbb{R})$ satisfying
  $$\varphi(s)=\begin{cases}
  -2d_0 &\;\;\;{\rm if\;\;} s\leq -2d_0,\\
    s & \;\;\;{\rm if\;\;}|s|\leq d_0, \\
    2d_0 & \;\;\;{\rm if\;\;}s\geq 2d_0.
  \end{cases}$$
If $d_0$ is sufficiently small, then, for any $0<\ep<\ep _0$,
$$
d^\ep(x,t):=\varphi(\tilde d ^\ep(x,t))
$$
is smooth in $\Omega\times(0,\mathcal T)$, satisfies $d^\ep(x,t)=0$ for
$x\in\tildet$,
\begin{equation}\label{nablad}
|\nabla d^\ep(x,t)|=1 \quad \text{ in } \{(x,t):\,|d^\ep(x,t)|<
d_0\}.
\end{equation}
Also, since the inward normal velocity $V$ and the mean curvature
$\kappa$ are equal to $\partial _t d^\ep$ and $\frac{\Delta d^\ep}{n-1}$,
equation \eqref{law-motion} is recast as
\begin{equation}\label{recast}
\partial _t d^\ep(y,t)=\Delta d^\ep (y,t)-\frac {c(\ep \xi^\ep
(t))}\ep\quad \text{ on } \{(y,t):\,y\in\tildet \}.
\end{equation}

\subsection{An $\mathcal O (\ep)$-sandwich of the layers}\label{ss:sandwich}

Equipped with the above material, we are now in the position to
construct sub-and supersolutions for equation \eqref{AC} in the
form
\begin{equation}\label{sub}
u_\ep^{\pm}(x,t):=m\left(\frac{d^\ep(x,t)\pm\ep p(t)}\ep;\ep
\xi^\ep(t)\right)\pm q(t),
\end{equation}
where
\begin{equation}\label{def:pq}p(t):=-\EB+e^{Lt}+ K, \quad
q(t):=\sigma \big( \beta \EB+\ep^2 Le^{Lt}\big),
\end{equation}
where $\beta$, $\sigma$, $K$ and $L$ are positive constants to be
chosen. Notice that $q=\sigma\ep^2\,p_t$.  Notice also that,
initially, the vertical shift $p(0)$ is $\mathcal O(1)$ but, as
soon as $t>0$, $p(t)$ becomes $\mathcal O(\ep ^2)$. Furthermore,
it is clear from the definition of $u_\ep^\pm$ that, as soon as
$t>0$, $\lim_{\ep\rightarrow 0} u_\ep^\pm(x,t)= a_-$, respectively
$a_+$, if $x\in\Omega^\ep _t$, respectively $x\in\Omega\setminus
\overline{\Omega^\ep _t}$.

\begin{prop}[Sub- and supersolutions for the propagation]\label{prop:subsuper}
Choose $\beta>0$ and $\sigma>0$ appropriately. Then for any $K>1$,
there exist constants $\ep_0>0$ and $L>0$ such that, for any
$\ep\in(0,\ep _0)$, the functions $(u_\ep^-,u_\ep^+)$ are a pair
of sub- and super-solutions for equation \eqref{AC} in the domain
$\Omega\times (0,\mathcal T)$, that is
$$ \mathcal L u_\ep^+:=
\partial _t u_\ep ^+-\Delta u_\ep^+ -\frac{1}{\ep^2} f(u_\ep^+)-\frac{1}{\ep}
\xi^\ep(t)\geq 0, \quad \mathcal L u_\ep^-\leq 0, $$ in
$\Omega\times (0,\mathcal T)$.
\end{prop}

\begin{proof} We only give the proof of the inequality for $u_\ep^+$, since the
one for $u_\ep^-$ follows the same argument. In the sequel, $m$
and its derivatives are evaluated at $$(z^*;\delta
^*):=\left(\frac{d^\ep(x,t)+\ep p(t)}\ep;\ep\xi^\ep(t)\right)
$$
which belongs to $\R \times (-\delta_0,\delta_0)$ if $\ep >0$ is
small enough. Straightforward computations combined with
$$
f(m+q)=f(m)+qf'(m)+\frac 12 q^2 f''(\theta), \quad \text{ for some
} m<\theta=\theta(x,t)<u_\ep^+,
$$
and equation \eqref{n1n} yield $\mathcal L u_\ep
^+=E_1+E_2+E_3+E_4$, with
\begin{eqnarray*}
E_1&=&-\frac 1{\ep^2} q\left(f'(m)+\frac 12 q f''(\theta)\right)+m_zp_t+q_t\\
E_2&=&(1-|\nabla d^\ep|^2)\frac{m_{zz}}{\ep ^2}\\
E_3&=&\left(\partial _t d^\ep(x,t)-\Delta d^\ep (x,t)+\frac {c(\ep
\xi^\ep(t))}\ep\right)\frac{m_z}\ep\\
E_4&=&\ep \dot \xi ^\ep(t)m_\delta.
\end{eqnarray*}

Let us first present some useful inequalities. By assumption
\eqref{cf1}, there are $b>0$, $\rho
>0$
 such that
\begin{equation}\label{bords}
f'(m(z;\delta))\leq -\rho \qquad \hbox{if} \quad m(z;\delta)\in
[a_--b,a_-+b]\cup[a_+-b,a_++b].
\end{equation}
On the other hand, since the region
$\{(z;\delta)\in\R\times(-\delta_0,\delta_0):\,m(z;\delta)\in
[a_-+b,\,a_+-b] \,\}$ is compact, there is  $a_1>0$ such that
\begin{equation}\label{milieu}
m_z(z;\delta) \geq a_1 \qquad\hbox{if} \quad m(z;\delta)\in
[a_-+b,\,a_+-b].
\end{equation}
We now select
\begin{equation}\label{betasigma}
\beta=\frac \rho 4, \quad 0<\sigma\leq \min(\sigma _0,\sigma
_1,\sigma_2),
\end{equation}
 where
$$
\sigma _0:=\frac{a_1}{\rho+\Vert f'\Vert
_{L^\infty(a_--1,a_++1)}},\quad \sigma _1:=\frac 1{2(\beta +1)},
\quad \sigma _2:=\frac{4\beta}{\Vert f''\Vert
_{L^\infty(a_--1,a_++1)}(\beta+1)}.
$$
Combining \eqref{bords}, \eqref{milieu} and  $0<\sigma\leq \sigma
_0$, we obtain
\begin{equation}\label{U0-f}
m_z(z;\delta)-\sigma f'(m(z;\delta))\geq \sigma \rho, \quad
\forall(z;\delta)\in\R \times (-\delta _0,\delta _0).
\end{equation}

Now let $K>1$ be arbitrary. In what follows we will show that
$\mathcal L u _\ep ^+ \geq 0$ provided that the constants $\ep_0$
and $L$ are appropriately chosen. We go on under the following
assumption (to be checked at the end)
\begin{equation}\label{ep0M}
\ep _0^2 Le^{L\mathcal T} \leq 1.
\end{equation}
Then, given any $\ep\in(0,\ep_0)$, we have, since $\sigma \leq
\sigma _1$, $0\leq q(t)\leq \frac 1 2$, that
\begin{equation}\label{uep-pm}
a_--1\leq u_\ep ^\pm(x,t) \leq a_++1.
\end{equation}

Using the expressions for $p$ and $q$, the \lq\lq favorable'' term $E_1$ is
recast as
$$
E_1=\frac{\beta}{\ep^2}\,\EB(I-\sigma\beta)+Le^{Lt}(I+\ep^2\sigma
L),
$$
where
$$
I=m_z(z^*;\delta^*)-\sigma f '(m(z^*;\delta^*))-\frac {\sigma^2}2
f ''(\theta)(\beta\EB+\ep^2 Le^{Lt}).
$$
In virtue of \eqref{U0-f}, \eqref{uep-pm} and \eqref{ep0M}, we
have $I\geq \sigma \rho-\frac {\sigma^2}{2} \Vert f''\Vert
_{L^\infty(a_--1,a_++1)}(\beta+1)$. Since $0<\sigma \leq \sigma
_2$, we obtain $ I \geq 2\sigma\beta$. Consequently, we have
$$
E_1\geq \frac{\sigma\beta^2}{\ep^2}\EB + 2\sigma\beta L e^{Lt}\geq
2\sigma \beta L e^{Lt}.
$$

Next, in view of \eqref{nablad}, $E_2=0$ in the region
$|d^\ep(x,t)|\leq d_0$. Next we consider the region where
$|d^\ep(x,t)|\geq d_0$. We deduce from Lemma
\ref{lem:tw-estimates} that
$$
|E_2|\leq \frac C{\ep^2}e^{-\lambda|d^\ep(x,t)+\ep p(t)|/ \ep}\leq
\frac{C}{\ep^2}e^{-\lambda(d_0 / \ep-p(t))}.
$$
We remark that $0<K-1 \leq p \leq e^{L\mathcal T} +K$.
Consequently, if we assume (to be checked at the end)
\begin{equation}\label{ga}
e^{L\mathcal T}+K \leq \frac{d_0}{2\ep_0},
\end{equation}
then $\displaystyle{\frac{d_0}{\ep}}-p(t)\geq
\displaystyle{\frac{d_0}{2\ep}}$, so that $ |E_2|\leq
\frac{C}{\ep^2}e^{-\lambda d_0 / (2\ep)}=\mathcal O(1)$, as $\ep
\to 0$.

Let us now turn to the term $E_3$. In the region where $\vert
d^\ep(x,t)\vert \geq\min(d_0,\frac 1{2 \mathcal K})>0$ (away from
the interface), argument similar as those for $E_2$ yield $\vert
E_3\vert =\mathcal O(1)$ as $\ep \to 0$ (thanks to the exponential
decay of the wave). In the region where $\vert d^\ep(x,t)\vert
\leq \min(d_0, \frac 1{2 \mathcal K})$, let us pick a $y\in\gamma
_t ^{\ep}$ such that $\vert d^{\ep}(x,t)\vert =dist(x,y)$. In view
of \eqref{recast} and $\partial _t d^\ep(x,t)=\partial _t
d^\ep(y,t)$ we get
$$
E_3=\left(\Delta d^\ep(y,t)-\Delta
d^\ep(x,t)\right)\frac{m_z}{\ep}.
$$
But it follows from \cite[Lemma 14.17] {Gil-Tru} that
\begin{eqnarray*}
\vert \Delta d^\ep(y,t)-\Delta d^\ep(x,t) \vert &=&\left\vert \sum _{i=1}^{n-1}\kappa ^\ep _i(y,t)-\sum _{i=1}^{n-1}\frac{\kappa ^\ep _i(y,t)}{1-d^{\ep}(x,t)\kappa ^\ep _i(y,t)}\right\vert\\
&\leq & \vert d^{\ep}(x,t)\vert \sum _{i=1}^{n-1}\frac{(\kappa ^\ep _i )^{2}(y,t)}{\vert 1-d^{\ep}(x,t)\kappa ^\ep  _i (y,t)\vert}\\
&\leq & 2\vert d^{\ep}(x,t)\vert \sum _{i=1}^{n-1}(\kappa ^\ep _i)
^{2}(y,t)
\end{eqnarray*}
since $\vert d^{\ep}(x,t)\vert \leq \frac{1}{2\mathcal K}$, and
$\vert \kappa ^\ep _i(y,t)\vert \leq \mathcal K$. As a result we have
$\vert E_3\vert \leq 2(n-1)\mathcal K ^{2}\vert
d^{\ep}(x,t)\vert=:C\vert d^{\ep}(x,t)\vert$, so that
\begin{eqnarray*}
\vert E_3\vert &\leq& C\frac{\vert d^\ep(x,t)\vert}{\ep}m_z\left(\frac{d^\ep(x,t)+\ep p(t)}{\ep};\ep\xi ^\ep(t)\right)\\
&\leq & C \sup_{z\in\R, \vert \delta\vert \leq \delta _0}\vert zm_z(z;\delta)\vert+C\ep\vert p(t)\vert \sup _{z\in\R, \vert \delta\vert \leq \delta _0}\vert m_z(z;\delta)\vert\\
& \leq &C_3+C_3'(e^{Lt}+K),
\end{eqnarray*}
for some constants $C_3>0$, $C_3'>0$ and where we have used Lemma
\ref{lem:tw-estimates}.

Last, it follows from \eqref{estimates}, \eqref{estimates2} and
Lemma \ref{lem:tw-estimates} that $|E_4|\to 0$, as $\ep \to 0$,
uniformly in $\Omega \times (0,\mathcal T)$.

Putting the above estimates all together, we arrive at $$ \mathcal
L u_\ep ^+\geq (2\sigma \beta L-C_3')e^{Lt} -\mathcal O(1)
$$ which is
nonnegative, if $L>0$ is sufficiently large, and $\ep _0>0$
sufficiently small to validate assumptions \eqref{ep0M} and
\eqref{ga}. The theorem is proved.\end{proof}

\section{Description of the $\mathcal O(\ep)$ layers and their convergence}\label{s:layers}

\subsection{Proof of Theorem \ref{th:main}}\label{ss:proof-main} Let $\eta\in(0,\eta _0)$ be given.
Let us select $\beta >0$ and $\sigma >0$ that satisfy
\eqref{betasigma}---so that Proposition \ref{prop:subsuper} is
available---and $\beta \sigma \leq \eta /3$. By the emergence of
the layers property, we are equipped with small $\ep _0>0$ and a
$M_0>0$ such that \eqref{g-part1}, \eqref{g-part2},
\eqref{g-part3} hold with $\beta\sigma /2$ playing the role of
$\eta$. On the other hand, in view of \eqref{gradientep}, there
is  $M_1>0$ such that we have the following correspondence
\begin{equation}\label{corres}
\begin{array}{ll}\text { if } \quad d^{\ep} (x,0) \geq \ M_1 \ep
&\text { then } \quad u_0(x) \geq a^{\ep} +M_0 \ep\vspace{3pt}\\
\text { if } \quad d^{\ep} (x,0) \leq -M_1 \ep & \text { then }
\quad u_0(x) \leq a^{\ep} -M_0 \ep,
\end{array}
\end{equation}
where we recall that  $d^{\ep}(x,0)$ denotes the signed distance
function associated with the hypersurface $\gamma _0
^{\ep}:=\{x:\, u_0(x)=a^{\ep}\}$. Now we define functions $H^+(x),
H^-(x)$ by
\[
\begin{array}{l}
H^+(x)=\left\{
\begin{array}{ll}
a_++\sigma\beta/2\quad\ &\hbox{if}\ \ d^\ep(x,0)\geq -M_1\ep\\
a_-+\sigma\beta/2\quad\ &\hbox{if}\ \ d^\ep(x,0)<  -M_1\ep,
\end{array}\right.\\
H^-(x)=\left\{
\begin{array}{ll}
a_+-\sigma\beta/2\quad\ &\hbox{if}\ \ d^\ep(x,0)\geq \;M_1\ep\\
a_--\sigma\beta/2\quad\ &\hbox{if}\ \ d^\ep(x,0)<  \;M_1\ep.
\end{array}\right.
\end{array}
\]
Then from the above observations we see that, after a very short
time $\mathcal O(\ep ^2\vert \ln\ep\vert)$, we have an $\mathcal
O(\ep)$ sandwich of the layers, namely
\begin{equation}\label{H-u}
H^-(x) \,\leq\, u^\ep(x,\mu_\ep ^{-1} \ep^2|\ln \ep|) \,\leq\,
H^+(x)\qquad \hbox{for}\ \ x\in\Omega.
\end{equation}

We now would like to use the sub and supersolutions \eqref{sub}
for the propagation described at Section \ref{s:motion}.
Observe that
$$
u_\ep ^\pm(x,0)=m\left(\frac{d^\ep(x,0)\pm
K}{\ep};\ep\xi^\ep(0)\right)\pm\sigma (\beta+\ep^2L),
$$
so that it follows from $\ep \xi ^\ep(0)=\mathcal O(\ep
^{1-\gamma})\to 0$ and Lemma \ref{lem:tw-estimates} on traveling
waves $m(z;\delta)$ that we can select $K>>M_1$ so that
\begin{equation*}
u_\ep^{-}(x,0)\leq H^-(x) \,\leq\, u^\ep(x,\mu_\ep ^{-1} \ep^2|\ln
\ep|) \,\leq\, H^+(x)\leq u_\ep ^{+}(x,0)\qquad \hbox{for}\ \
x\in\Omega.
\end{equation*}
Let us now choose $\ep _0>0$ and $L>0$ so that Proposition
\ref{prop:subsuper} applies. It therefore follows from the
comparison principle that
\begin{equation}\label{ok}
u_\ep^-(x,t) \leq u^\ep (x,t+t^\ep) \leq u_\ep^+(x,t) \quad \text
{ for } x\in\Omega, \,0 \leq t \leq \mathcal T-t^\ep,
\end{equation}
where $t^\ep=\mu _\ep ^{-1} \ep ^2|\ln \ep|$.

To conclude, in view of $\ep\xi^\ep(t)=\mathcal
O(\ep^{1-\gamma})\to 0$ and Lemma \ref{lem:tw-estimates} on
traveling waves, we can select $\ep _0>0$ small enough and $C>0$
large enough so that, for all $\ep\in(0,\ep _0)$, all $0\leq t\leq
\mathcal T-t^\ep$,
\begin{equation}\label{C}
m(C-e^{L\mathcal T}-K;\ep\xi^\ep(t)) \geq a_+-\frac \eta 2 \quad \text {
and } \quad m(-C+e^{L\mathcal T}+K;\ep\xi ^\ep(t)) \leq a_-+\frac \eta 2.
\end{equation}
Using inequalities \eqref{ok}, expressions \eqref{sub} for $u_\ep
^\pm$, estimates \eqref{C} and $\sigma\beta\leq \eta /3$ we  then
see that, for all $\ep\in(0,\ep _0)$ and all $0\leq t\leq
\mathcal T-t_\ep$, we have
\begin{equation}\label{correspon}
\begin{array}{ll}\text { if } \quad d^\ep(x,t) \geq \ C \ep &
\text { then } \quad
u^\ep (x,t+t^\ep) \geq a_+ -\eta\vspace{3pt}\\
\text { if } \quad d^\ep(x,t) \leq -C \ep & \text { then } \quad
u^\ep (x,t+t^\ep) \leq a_- +\eta,
\end{array}
\end{equation}
and $ u^\ep (x,t+t^\ep) \in [a_--\eta,a_++\eta], $ which completes
the proof of Theorem \ref{th:main}. \qed

\subsection{Proof of Corollary \ref{cor:funaki}}\label{ss:stochastic-fun} Let us observe that the simplifying
assumption $\xi ^\ep(0)=0$ enables to get rid of the initial small drift which happens during the emergence
of the layers. Precisely, in view of \eqref{crucial} and \eqref{def:shifted-initial-interface}, $\gamma _0 ^\ep$
is nothing else than $\Gamma _0$. As a result, the approximated (deterministic) $\gamma ^\ep _t$ involve perturbations
of the speed but not of the initial data. This enables (see the end of subsection \ref{ss:mmc-fun}) to reproduce
the arguments of \cite{Fun99} to derive Corollary \ref{cor:funaki} from our Theorem \ref{th:main}.

Notice that, if $\xi ^\ep  (0)\neq 0$, then one needs to derive
an analogous of \eqref{cor-dirr} (available in the Weber's
context) in the Funaki's context. We think that, following
\cite{Fun99}, this can be performed but this is beyond the scope
of the present paper so we decided to avoid this situation.

\subsection{Proof of Corollary \ref{cor:weber}}\label{ss:stochastic-web} Combining Theorem \ref{th:main}
and estimate \eqref{cor-dirr}, we  get Corollary \ref{cor:weber} by reproducing the arguments
of  \cite[Proof of Theorem 1.1] {weber1}.

\section{Profile in the layers}\label{s:profile}

Equipped with Theorem \ref{th:main}, we can now prove the validity
of the first term of the asymptotic expansions {\it inside} the layers,
namely Theorem \ref{th:profile}. The proof consists in using the
stretched variables, a blow-up argument and the result of
\cite{Ber-Ham}, as performed in the deterministic case \cite{am}.

Before going further, we recall that a solution of an evolution
equation is called {\it eternal} (or an {\it entire} solution) if
it is defined for all positive and negative time. We follow this
terminology to refer to a solution $w(z,\tau)$ of
\begin{equation}\label{eq-eternal}
w_{\tau} = \Delta _z w + f(w)\,, \quad z\in\R^n,\,\tau \in\R\,.
\end{equation}
Stationary solutions and travelling waves are examples of eternal
solutions. We quote below a result of Berestycki and Hamel
\cite{Ber-Ham} asserting that \lq\lq any planar-like eternal
solution is actually a planar wave". More precisely, the following
holds (for $z\in \R ^n$ we write $z=(z^{(1)},\cdots,z^{(n)})$).

\begin{lemma}[{\cite[Theorem 3.1]{Ber-Ham}}]\label{lem:Ber-Ham}
Let $w(z,\tau)$ be an eternal bounded solution of
\eqref{eq-eternal} satisfying
\begin{equation}\label{conditions}
\liminf_{z^{(n)}\to \infty}\; \inf_{z'\in\R^{n-1},\,\tau\in \R}
w(z,\tau)
> a\,, \quad \limsup_{z^{(n)}\to-\infty}\; \sup_{z'\in\R^{n-1},\,\tau\in\R}
w(z,\tau) < a\,,
\end{equation}
where $z':=(z^{(1)},\cdots,z^{(n-1)})$. Then there exists  a
constant $z^*\in\R$ such that
\[
w(z,\tau) = U_0(z^{(n)}-z^*)\,,  \quad z\in\R^n\,,\,\tau\in\R\,.
\]
\end{lemma}

\subsection{Proof of (ii) in Theorem \ref{th:profile}}
\label{SS:proof1} Let $\rho>1$ and $0<T'<T$ be given. Assume by
contradiction that \eqref{lim-d-epsilon} does not hold. Then there
is $\eta >0$ and sequences $\ep _k \downarrow 0$, $t_k\in[\rho
t^{\ep _k},T']$, $x_k \in \bar \Omega$ ($k=1,2,...$) such that
\begin{equation}\label{eta}
  \left |u^{\ep _k}(x_k,t_k)-U_0\left (\frac{\overline{d^{\ep_k}}(x_k,t_k)}
  {\ep_k}\right)\right|\geq 2\eta\,.
\end{equation}
In view of \eqref{resultat}, \eqref{dist-ep} and $U_0(\pm
\infty)=a_\pm$, for \eqref{eta} to hold it is necessary to have
\begin{equation}\label{near}
  d^\ep(x_k,t_k)=\mathcal O (\ep _k)\,,\quad \text{ as }\, k\to
  \infty\,.
\end{equation}
Recall that $d^{\ep}(\cdot,t)$ denotes the signed distance function to $\gamma _t  ^{\ep}$ as
defined in subsection \ref{ss:preliminaries2}, whereas $\overline{d^{\ep}}(\cdot,t)$
denotes that to $\Gamma _t ^{\ep}$ defined in \eqref{eq:dist-e}.

If $u^{\ep _k}(x_k,t_k)=a$, then this would mean that $x_k\in
\Gamma^{\ep_k}_{t_k}$, in which case the left-hand side of
\eqref{eta} would be $0$ (since $U_0(0)=a$), which is impossible.
Hence $u^{\ep _k}(x_k,t_k)\ne a$. By extracting a subsequence if
necessary, we may assume without loss of generality that $u^{\ep
_k}(x_k,t_k)- a$ has a constant sign for $k=0,1,2,\ldots$, say
\begin{equation}\label{u<0}
  u^{\ep _k}(x_k,t_k)> a \quad (k=0,1,2,\ldots),
\end{equation}
which then implies that $\overline{d^{\ep_k}}(x_k,t_k)> 0$
($k=0,1,2,\ldots$).  Since  the mean curvature of
$\gamma _t^\ep$ is uniformly bounded for $0\leq t \leq T'$,
$0<\ep<<1$, there is a small $\delta
>0$ such that each $x$ in a $\delta$-tubular neighborhood of
$\gamma _t^\ep$ has a unique orthogonal projection on $\gamma
_t^\ep$. Since the sequence $(x_k)$ remains very close to
$\gamma_{t_k}^{\ep _k}$ by \eqref{near}, each $x_k$ (with
sufficiently large $k$) has a unique orthogonal projection
$p_k=p^{\ep_k}(x_k,t_k) \in \gamma _{t_k}^{\ep _k}$. Let $y_k$ be
a point on $\Gamma^{\ep_k}_{t_k}$ that has the smallest distance
from $x_k$. If such a point is not unique, we choose one such
point arbitrarily. Then we have
\begin{equation}\label{u=0}
  u^{\ep_k}(y_k,t_k)=a \quad (k=0,1,2,\ldots),
\end{equation}
\begin{equation}\label{d=x-y}
  \overline{d^{\ep_k}}(x_k,t_k)=\Vert x_k - y_k\Vert\,,
\end{equation}
\begin{equation}\label{circle}
  u^{\ep_k}(x,t_k)>a \quad\  \hbox{if}\ \ \Vert x-x_k\Vert
  < \Vert y_k - x_k\Vert\,.
\end{equation}
\[
x_k-p_k\perp \gamma_{t_k}^{\ep _k} \quad\hbox{at}\ \
p_k\in\gamma_{t_k}^{\ep _k}\,,
\]
 Furthermore, \eqref{near} and
\eqref{dist-ep} imply
\begin{equation}\label{near2}
  \Vert x_k-p_k\Vert = {\mathcal O}(\ep_k)\,, \ \
  \Vert y_k-p_k\Vert = {\mathcal O}(\ep_k) \quad
  (k=0,1,2,\ldots).
\end{equation}

We now rescale the solution $u^\ep$ around $(p_k,t_k)$ and define
\begin{equation}\label{rescaling}
  w^k(z,\tau):=u^{\ep _k}(p_k+\ep _k \mathcal R _k z,t_k+\ep _k ^2
  \tau)\,,
\end{equation}
where $\mathcal R _k$ is a matrix in $SO(n,\R)$ that rotates the
$z^{(n)}$ axis onto the normal at $p_k\in \gamma _{t_k}^{\ep _k}$,
that is,
\[
\mathcal  R_k:(0,\dots,0,1)^T\mapsto n^{\ep _k}(p_k,t_k)\,,
\]
where $(\ )^T$ denotes a transposed vector and $n^\ep(p,t)$ the
outward normal unit vector at $p\in \gamma_{t}^\ep$. Since
$\gamma_t^\ep$ (hence the points $p_k$) is uniformly separated
from $\partial\Omega$ by some positive distance, there exists
$c>0$ such that $w^k$ is defined (at least) on the box
\[
 B^k:=\left \{(z,\tau)\in \R^n\times \R \,:\, \Vert z\Vert \leq
 \frac c{\ep _k}, \ \  -(\rho -1)\mu _{\ep_k}^{-1}|\ln \ep
_k|\leq \tau \leq \frac{T-T'}{\ep _k ^2}\right\}\,,
\]
where we recall that $\mu_\ep \to \mu=f'(a)>0$ as $\ep \to 0$.
Since $u^\ep$ satisfies the equation in \eqref{AC}, we see that
$w^k$ satisfies
\begin{equation}\label{edp-wk}
  w^k _\tau=\Delta _z w^k+f(w^k)+\ep_k \xi ^{\ep _k}(t_k+\ep _k ^2 \tau)\quad \text{ in } B^k\,.
\end{equation}
Moreover, if $(z,\tau)\in B^k$ then $t^{\ep_k}\leq t_k+\ep _k ^2
\tau \leq T$. Therefore \eqref{resultat} implies
\begin{equation}\label{front-like0}
  \left\{ \begin{array}{lll}
  d^{\ep _k}(p_k+\ep _k \mathcal R _k z,t_k+\ep _k^2\tau) \leq -C\ep _k
   \ \ & \Rightarrow
  \ & w^k(z,\tau) \leq a _-+\eta\,,\vspace{4pt}\\
  d^{\ep _k}(p_k+\ep _k \mathcal R _k z,t_k+\ep _k^2\tau) \geq C\ep _k
   \ \ & \Rightarrow
  \ & w^k(z,\tau) \geq a_+-\eta\,,
  \end{array}\right.
\end{equation}
as long as $(z,\tau) \in B^k$. Now we recall that the
rotation by $\mathcal R _k$ of the $z^{(n)}$ axis is normal to
$\gamma_{t_k}^{\ep _k}$ at $p_k$, and that the mean curvature of
$\gamma_t ^\ep$ is uniformly bounded for $0\leq t\leq T'$,
$0<\ep<<1$. Also the normal speed of $\gamma _t ^\ep$, given by
$V=(n-1)\kappa -\frac{c(\ep \xi _t^\ep)}{\ep}$,  is $\mathcal
O(\ep^{-\gamma '})$ for some $0<\gamma '<\frac 13$ in view of
$c(\delta)=-c_0 \delta +\mathcal O(\delta ^{2})$ as $\delta \to
0$, and \eqref{estimates} (if (MN1) noise) or Proposition
\ref{prop:Weber-noise} (if (MN2) noise). As  a result
$d^{\ep}(x,t)$ satisfies
$$
\vert d^{\ep}(x,t)-d^{\ep}(x,t')\vert \leq \frac{\tilde C}{\ep ^{\gamma'}}\vert t -t'\vert, \quad 0\leq t,t'\leq T', 0<\ep<<1,
$$
for some $\tilde C>0$. From these observations and
\eqref{front-like0}, we see that there exists a constant $K>0$,
which is independent of $k$, such that
\begin{equation}\label{front-like}
  z^{(n)}\leq -K \Rightarrow w^k(z,\tau)\leq a _-+\eta\,,\quad
   z^{(n)}\geq K \Rightarrow w^k(z,\tau)\geq a_+-\eta\,,
\end{equation}
for all $(z,\tau) \in B^k$ with $\Vert z\Vert\leq\sqrt{1/\ep_k}$
and $|\tau|\leq 1/({\ep _k}^{1-\gamma'})$.

Now, since $w^k$ solves \eqref{edp-wk}, the uniform (w.r.t. $k\geq
0$) boundedness of $w^k$ and standard parabolic estimates, along
with the derivative bounds on $\tau \mapsto \ep_k \xi ^{\ep
_k}(t_k+\ep _k ^2 \tau)$ (see \eqref{estimates} or Proposition \ref{prop:Weber-noise}), imply that $w^k$ is
uniformly bounded in $C_{loc}^{2+\gamma,1+\frac \gamma 2}(B^1)$.
We can therefore extract from $(w^k)$ a subsequence that converges
to some $w$ in $C_{loc}^{2,1}(B^1)$. By repeating this on all
$B^k$, we can find a subsequence of $(w^k)$ that converges to some
$w$ in $C^{2,1}_{loc}(\R^n\times \R)$ (note that $\cup_{k\geq 0}
B^k=\R^n \times \R$). Passing to the limit in \eqref{edp-wk}
yields
\[
 w_\tau=\Delta _zw +f(w)\quad \text{ on } \R^n\times \R\,.
\]
Hence we have constructed an eternal solution $w(z,\tau)$ which---in view of \eqref{front-like}---satisfies
\eqref{conditions}. Lemma \ref{lem:Ber-Ham} then implies that
\begin{equation}\label{w=U0}
w(z,\tau)=U_0(z^{(n)}-z^*)
\end{equation}
for some $z^*\in\R$.

Now we define sequences of points $(z_k),\,(\tilde z_k)$ by
\[
 z_k:=\frac 1 {\ep _k} \mathcal R _k ^{-1}(x_k-p_k)\,,\quad
 \tilde z_k:=\frac 1 {\ep _k} \mathcal R _k ^{-1}(y_k-p_k)\,.
\]
By \eqref{near2}, these sequences are bounded, so we may assume
without loss of generality that they converge:
\[
  z_k\to z_\infty\,, \quad \tilde z_k\to \tilde z_\infty\,, \quad
  \hbox{as}\ \ k\to\infty\,.
\]
By the definition of the $z$ coordinates, $z_\infty$ must lie on
the $z^{(n)}$ axis, that is,
\[
  z_\infty=(0,\dots,0,z^{(n)}_\infty)^T\,.
\]
It follows from \eqref{u=0} and \eqref{circle} that
\begin{equation}\label{circle-w}
  w(\tilde z_\infty,0)=a\,,\quad\ \  w(z,0)\geq a \ \ \hbox{if}\ \
  \Vert z-z_\infty\Vert \leq     \Vert\tilde z_\infty-z_\infty\Vert\,.
\end{equation}
Note that by \eqref{w=U0}, the level set $w(z,0)=a$ coincides with
the hyperplane $z^{(n)}=z^*$, and recall that ${U_0}'>0$.
Therefore, in view of \eqref{w=U0} and \eqref{circle-w}, we have
either $\tilde z_\infty=z_\infty$, or that the ball of radius
$\Vert\tilde z_\infty-z_\infty\Vert$ centered at $z_\infty$ is
tangential to the hyperplane $z^{(n)}=z^*$ at $\tilde z_\infty$.
This implies that $\tilde z_\infty$, as well as $z_\infty$, must
also lie on the $z^{(n)}$ axis. Therefore
\[
\tilde z_\infty=(0,\dots,0,z^*)^T\,,
\]
and the inequality $w(z_\infty,0)\geq a$ implies that
$z_\infty^{(n)}\geq z^*$. On the other hand equality \eqref{d=x-y} implies
$\overline{d^{\ep_k}}(x_k,t_k)/\ep_k=\Vert
x_k-y_k\Vert/\ep_k=\Vert z_k-\tilde z_k\Vert\to \Vert
z_\infty-\tilde z_\infty \Vert =z^{(n)}_\infty-z^*$. The
assumption \eqref{eta} then yields
\begin{eqnarray*}
0&=& \left| w(z_\infty,0)-U_0(z_\infty^{(n)}-z^*)\right|\vspace{2pt}\\
&=&\left|\lim_{k\to\infty}
u^{\ep_k}(x_k,t_k)-U_0\Big(\lim_{k\to\infty}\frac{\overline{d^{\ep_k}}(x_k,t_k)}
{\ep_k}\Big)\right|\vspace{2pt}\\
& \geq& 2\eta\,.
\end{eqnarray*}
This contradiction proves statement (ii) of Theorem
\ref{th:main}.\qed

\subsection{Proof of (i) in Theorem \ref{th:profile}}
The proof of (i) below uses an argument similar to the proof of
Corollary 4.8 in \cite{Mat-Nar}. Fix $\rho >1$ and $0<T'<T$. For a
given $\eta \in (0,\min(a-a _-,a_+ -a))$ define $\ep _0 >0$ and
$C>0$ as in Theorem \ref{th:main}. Then we claim that
\begin{equation}\label{claim}
  \liminf _{\ep \to 0} \inf _{x\in \mathcal N_{C \ep}(\gamma
  _t ^\ep),\, \rho t^\ep \leq t \leq T'} \nabla u^\ep (x,t) \cdot
  n^\ep(p(x,t),t)>0\,,
\end{equation}
where $n^\ep(p,t)$ denotes the outward unit normal vector at $p\in
\gamma _t ^\ep$.

Indeed, assume by contradiction that there exist sequences $\ep
_k \downarrow 0$, $t_k\in[\rho t^{\ep _k},T']$, $x_k \in \mathcal
N_{C \ep _k}(\gamma _{t_k}^{\ep _k})$ ($k=1,2,...$) such that
\[
  \nabla u^{\ep _k} (x_k,t_k) \cdot n^{\ep _k}(p_k,t_k) \leq 0\,,
\]
where $p_k=p(x_k,t_k)$. By rescaling around $(p_k,t_k)$ and using
arguments similar to those in the proof of (ii), one can find a
point $z_\infty$ with $|z_\infty ^{(n)} |\leq C$ such that
\[
  {U_0}'(z_\infty ^{(n)})\leq 0\,,
\]
which contradicts to the fact that ${U_0}'>0$ and
establishes \eqref{claim}. Since, in view of Theorem
\ref{th:main}, $\Gamma _t ^\ep \subset \mathcal N _{C \ep}(\gamma
_t ^\ep)$, the estimate \eqref{claim} implies that $\nabla u^\ep
(x,t) \neq 0$ for all $x\in \Gamma _t ^\ep$; hence by the
implicit function theorem, $\Gamma _t ^\ep$ is a smooth
hypersurface in a neighborhood of any point on it. The fact that
$\Gamma _t ^\ep$ can be expressed as a graph over $\gamma _t
^\ep$ also follows from \eqref{claim}. This proves the
statement (i) of Theorem \ref{th:main}. \qed

\bigskip

\noindent {\bf Acknowledgments.} M. Alfaro is supported by
the ANR I-SITE MUSE, project MICHEL 170544IA (n$^{o}$
ANR-IDEX-0006). H. Matano is supported by KAKENHI (16H02151).


\end{document}